\documentclass[10 pt,a4paper]{article}

\usepackage{latexsym,amsfonts,amsmath,amsthm,amssymb, graphicx, mathrsfs, fancyhdr, makeidx, amscd,  color, mathdesign, fontenc, verbatim, stix}

\usepackage[all]{xy}

\newtheorem{theorem}{Theorem}[section]

\newtheorem{corollary}[theorem]{Corollary}

\newtheorem{definition}[theorem]{Definition}
\newtheorem{example}[theorem]{Example}

\newtheorem{lemma}[theorem]{Lemma}

\newtheorem{proposition}[theorem]{Proposition}
\newtheorem*{proposition-app}{Proposition}
\newtheorem{remark}[theorem]{Remark}

\newcommand{\R}{\mathbb{R}}

\newcommand{\metric}{\langle \, , \, \rangle}

\newcommand{\disp}{\displaystyle}

\newcommand{\ra}{\rightarrow}

\newcommand{\eps}{\varepsilon}

\newcommand{\II}{\mathrm{II}}
\newcommand{\Sph}{\mathbb{S}}
\newcommand{\di}{\mathrm{d}}

\newcommand{\Ricc}{\mathrm{Ric}}

\newcommand{\lip}{\mathrm{Lip}}
\newcommand{\loc}{\mathrm{loc}}

\newcommand{\tcr}{\textcolor{red}}
\newcommand{\tcb}{\textcolor{blue}}

\newcommand{\capac}{\mathrm{cap}}
\newcommand{\MM}{\mathscr{M}}

\newcommand{\LL}{\mathscr{L}}

\DeclareMathOperator{\diver}{div\,}
\DeclareMathOperator{\dist}{dist}

\begin{document}

\author{Giulio Colombo \and Luciano Mari \and Marco Rigoli}
\title{\textbf{A splitting theorem for capillary graphs under Ricci lower bounds}}
\date{}
\maketitle

%


\maketitle

\begin{abstract}
In this paper, we study capillary graphs defined on a domain $\Omega$ of a complete Riemannian manifold, where a graph is said to be capillary if it has constant mean curvature and locally constant Dirichlet and Neumann conditions on $\partial \Omega$. Our main result is a splitting theorem both for $\Omega$ and for the graph function on a class of manifolds with nonnegative Ricci curvature. As a corollary, we classify capillary graphs over domains that are globally Lipschitz epigraphs or slabs in a product space $N \times \R$, where $N$ has slow volume growth and non-negative Ricci curvature. A technical core of the paper is a new gradient estimate for positive CMC graphs on manifolds with Ricci lower bounds.

\end{abstract}

\tableofcontents

\newpage 


\vspace{1cm}

\noindent \textbf{Competing interests:} the authors declare that they have no conflict of interest.\\
\textbf{Declarations of interest:} none.

\section{Introduction} 

The study of capillary hypersurfaces in an ambient manifold with boundary is a classical subject, \cite{finn_book}, that in recent years stimulated a renewed interest from the mathematical community. A capillary hypersurface $\Sigma$ in a Riemannian manifold with boundary $\bar M$ is a constant mean curvature (CMC) hypersurface with boundary $\partial \Sigma \subset \partial \bar M$, that meets $\partial \bar M$ at a constant angle. For instance, the class includes free boundary minimal hypersurfaces in $\bar M$. Capillary hypersurfaces arise from a variational setup that we now recall, referring to \cite{finn_book,ros_souam} for a more detailed insight.\par 
Fix $\gamma \in (0, \pi/2]$, and suppose that $\Sigma$ is an embedded, connected hypersurface of $\bar M$, with boundary $\partial \Sigma \subset \partial \bar M$ that we assume to be transversal to $\partial \bar M$. Fix a connected, relatively compact open set $U \Subset \bar M$ in such a way that $\Sigma$ divides $U$ into two connected components $A$ and $B$. Let $B$ be the component inside which the angle $\gamma$ is computed. Note that $\partial \overline{B} \cap U$ is the union of a portion of $\Sigma$ and, possibly, of a relatively open subset $\Omega \subset \partial \bar M$. Then, $\Sigma$ is a stationary point for the functional
	\[
	|\Sigma \cap U| - \cos \gamma |\Omega|, 
	\]
with respect to variations that are compactly supported in $U$ and preserve the volumes of $A$ and $B$, if and only if $\Sigma$ is a CMC hypersurface satisfying 
	\[
	\langle \eta, \bar \eta \rangle = \cos \gamma \qquad \text{on } \, \partial \Sigma \cap U, 
	\]
where $\bar \eta$ and $\eta$ are, respectively, the outward pointing unit normals to $\partial \Sigma \hookrightarrow \Omega$ and $\partial \Sigma \hookrightarrow\Sigma$. \par
	Rigidity issues for capillary hypersurfaces have been investigated, to our knowledge, mostly in compact ambient spaces $\bar M$ with a large amount of symmetries. Especially, capillary hypersurfaces in the unit ball $\bar M = \mathbb{B}^{m+1}$ attracted the attention of researchers, also in view of the link between free boundary minimal hypersurfaces in $\mathbb{B}^{m+1}$ and the Steklov eigenvalue problem (cf. \cite{fraserschoen_1,fraserschoen_2} and the references therein). For compact $\Sigma \ra \mathbb{B}^{m+1}$, complete classification results were obtained under the assumptions that $\Sigma$ is stable, see \cite{ros_vergasta,nunes,barbosa} (free boundary case) and \cite{ros_souam,wangxia}.\par
	In the present paper, we study the rigidity problem for capillary hypersurfaces that are globally graphical over $\Omega$. More precisely, given a complete Riemannian manifold $(M^m,\sigma)$ without boundary, and a connected, open domain with smooth boundary $\Omega\subset M$, we let $\bar M = \overline{\Omega} \times \R^+_0$, where hereafter 
	\[
	\R^+_0 : = [0, \infty), \qquad \R^+ : = (0, \infty),
	\]
and consider $\Sigma$ to be the graph of a function
	\[
	u \ : \ \overline{\Omega} \ra [0, \infty) 
	\]
satisfying $u = 0$ on $\partial \Omega$. Having fixed $U \Subset \bar M$, we let $B$ denote the subgraph of $\Sigma$ in $U$. Then, $\Sigma$ is a capillary hypersurface if and only if $u$ satisfies the overdetermined boundary value problem
	\begin{equation}\label{eq_overdet}
	\left\{ \begin{array}{ll}
    \diver \left( \frac{Du}{\sqrt{1+|Du|^2}} \right) = H & \quad \text{on } \, \Omega \\[0.5cm]
	u = 0, \ \ \partial_{\bar \eta} u = - \tan \gamma & \quad \text{on } \, \partial \Omega,
	\end{array}\right.
	\end{equation}
for some constant $H \in \R$ being the (unnormalized) mean curvature of $\Sigma$. Here, as before, $\bar \eta$ is the unit exterior normal to $\partial \Omega \hookrightarrow \Omega$. \par

More generally, we consider solutions $u \in C^2(\overline\Omega)$ of the following overdetermined problem:

\begin{equation}\label{overdet}
	\left\{ \begin{array}{ll}
		\diver \left( \frac{D u}{\sqrt{1+ |D u|^2}}\right) = H & \qquad \text{on } \, \Omega, \\[0.5cm]
		u, \partial_{\bar\eta} u \, \text{ locally constant} & \qquad \text{on } \, \partial\Omega \\[0.3cm]
		\inf_\Omega u > -\infty,
	\end{array}\right.
\end{equation}
where $H\in\R$ is a given constant, whose graphs over $\Omega$ we still call capillary graphs. Our goal is to prove a classification result for both $\Omega$ and $u$ under mild assumptions on $M$. Precisely, we shall deduce from the existence of a solution of \eqref{overdet} both that $\Omega$ splits as a product $I \times N$, for some interval $I \subset \R$, and that $u$ only depends on the split $\R$-direction. In this respect, our main result is strongly inspired by the splitting theorems of J.Cheeger and D.Gromoll \cite{cheeger_gromoll} as well as those of P.Li and J.Wang, for manifolds with positive spectrum \cite{liwang2, liwang} and for properly immersed minimal hypersurfaces in manifolds of non-negative sectional curvature \cite{liwang_crelle}. However, the techniques described below depart in many aspects from those in \cite{cheeger_gromoll,liwang,liwang2}. In particular, the presence of a boundary introduces nontrivial further difficulties. 

To describe our main results, we first list three examples of domains supporting a solution of \eqref{overdet}, that we will hereafter consider as the \emph{standard examples}.

\begin{example}[\textbf{The standard examples}]\label{ex_standard}
\emph{For all of the cases, $\Omega = I \times N$ for an open interval $I \subset \R$ and a complete $(m-1)$-dimensional manifold $(N,\sigma_N)$, endowed with the product metric $\di t^2 + \sigma_N$. For $c_1 \le 0$ and $H \in \R$ define
	\begin{equation}\label{eq_split}
	u_{c_1,H}(t) = \left\{ \begin{array}{ll}
-c_1 t & \quad \text{if } \, H = 0 \\[0.3cm]
	\frac{1}{H} \left( \frac{1}{\sqrt{1+c_1^2}} - \sqrt{1- \left(Ht - \frac{c_1}{\sqrt{1+c_1^2}}\right)^2} \right) & \quad \text{if } \, H \neq 0.
	\end{array}\right.
 	\end{equation}
Then, we have:
\begin{itemize}
\item[(i)] \textbf{Half-hyperplane}. Here, $H = 0$, $\Omega = (0,\infty) \times N$ and $u(t,y) = u_{c_1,0}(t)$ for $c_1 < 0$. In this case, $\partial \Omega$ is connected.
\item[(ii)] \textbf{Piece of half-hyperplane}. Here, $H = 0$, $\Omega = (0,T) \times N$ and $u(t,y) = u_{c_1,0}(t)$, for some $c_1 < 0$. 	
\item[(iii)] \textbf{Piece of cylinder}. Here, $H \neq 0$, $\Omega = (0,T) \times N$ and $u(t,y) = u_{c_1,H}(t)$ for some $c_1 \le 0$ (with $c_1 < 0$ in case $H<0$). For reasons that will be clarified later, we shall restrict to intervals where $u$ is monotone increasing in $t$. Therefore, $T$ satisfies
			\[
			\disp T < \frac{1}{H}\left( 1 + \frac{c_1}{\sqrt{1+c_1^2}}\right) \ \ \ \text{if } \, H > 0, \qquad T \le \frac{1}{|H|}\frac{|c_1|}{\sqrt{1+c_1^2}} \ \ \ \text{if } \, H< 0.
			\]
\end{itemize}
}
\end{example}

\begin{remark}
	\emph{Evidently, in the case $T < \infty$, \eqref{eq_split} relates the constants describing the boundary data on $\partial_2 \Omega = \{T\} \times N$ and of $\partial_1\Omega = \{0\} \times N$. Furthermore, the only standard example with $u$ constant on $\partial\Omega$ is the half-hyperplane. 
	}
\end{remark}

Our goal is to prove that, under suitable assumptions, a domain with $\Ricc \ge 0$ and supporting a non-constant solution of \eqref{overdet} is a standard example, for some complete $N$ with $\Ricc_N \ge 0$. As a corollary of our main result, we will obtain the following classification of capillary domains that can be written as epigraphs or slabs bounded by globally Lipschitz functions, in the case when $N$ has slow volume growth. The next result applies to classify, for instance, globally Lipschitz epigraphs and slabs in $\R^2$, as well as globally Lipschitz bounded slabs in $\R^3$.

\begin{theorem}[\textbf{Lipschitz epigraphs and slabs}] \label{thm-epigraph}
Given $m \ge 2$, let $N$ be a complete, connected, boundaryless Riemannian manifold of dimension $m-1$. If $m \ge 3$, assume further that  
	\begin{equation}\label{ipo_epigraph}
	\Ricc \ge 0, \qquad \limsup_{r \to \infty} \frac{|B_r^N|}{r \log r} < \infty.
	\end{equation}
with $B_r^N$ the ball of radius $r$ in $N$ centered at a fixed origin. Define $M = \R \times N$, and let $\varphi_1,\varphi_2 \in C^\infty(N)$ be globally Lipschitz functions with $\varphi_1 < \varphi_2$ on $N$.
	\begin{itemize}
	\item[(i)] If the epigraph 
		\[
		\Omega : = \Big\{ (\tau,x) \in M \ : \ \tau > \varphi_1(x)\Big\} 
		\]
	supports a non-constant solution of 
	\[
	\begin{cases}
			\diver \left( \dfrac{Du}{\sqrt{1+|Du|^2}} \right) = H & \qquad \text{on } \, \Omega, \\
			u = 0, \ \partial_{\bar\eta} u = c \neq 0 & \qquad \text{on } \, \partial \Omega,\\
			\inf_\Omega u > -\infty,
		\end{cases}
	\]			
	for some $H,c \in \R$, where $\bar \eta$ is the unit exterior normal to $\partial \Omega \hookrightarrow \Omega$, then $H=0$, $c<0$ and one of the following cases occurs:
	\begin{itemize}
		\item [-] $\varphi_1$ is constant (up to translation, $\varphi_1\equiv 0$) and $u(\tau,x) = u_{c,0}(\tau)$, with $u_{c,0}$ as in \eqref{eq_split};
		\item [-] $N$ splits as a Riemannian product $N = \R \times N_0$ for some compact, boundaryless $N_0$ and, by denoting $x = (s,\xi) \in \R \times N_0$, we have
		$$
			\varphi_1(s,\xi) = a_0 s + a_1, \qquad u(\tau,s,\xi) = u_{c,0} \left( \frac{\tau - a_0 s - a_1}{\sqrt{1+a_0^2}} \right)
		$$
		for some constants $a_0,a_1 \in \R$ with $a_0\neq 0$.
	\end{itemize}
	\item[(ii)] If the slab
		\[
		\Omega : = \Big\{ (\tau,x) \in M \ : \ \varphi_1(x) < \tau < \varphi_2(x)\Big\} 
		\]
	supports a non-constant solution of 
	\[
	\begin{cases}
			\diver \left( \dfrac{Du}{\sqrt{1+|Du|^2}} \right) = H & \qquad \text{on } \, \Omega, \\
			u = 0, \ \partial_{\bar\eta} u = c_1 \le 0 & \qquad \text{on } \, \{t = \varphi_1(x)\} \\
			u = b > 0, \ \partial_{\bar\eta} u = c_2 \ge 0 & \qquad \text{on } \, \{t = \varphi_2(x)\} \\
			\inf_\Omega u > -\infty,
		\end{cases}
	\]			
	for some $H,b,c_1,c_2 \in \R$ with $(c_1, c_2) \neq (0,0)$, where $\bar \eta$ is the unit exterior normal to $\partial \Omega \hookrightarrow \Omega$, then one of the following cases occurs:
	\begin{itemize}
		\item [-] $\varphi_1 \equiv a_1$ and $\varphi_2 \equiv a_2$ are constant, and $u(\tau,x) = u_{c_1,H}(\tau - a_1)$, where $c_1 < 0$ if $H \le 0$;
		\item [-] $N$ splits as a Riemannian product $N = \R \times N_0$ for some compact, boundaryless $N_0$ and, denoting $x = (s,\xi) \in \R \times N_0$, we have $\varphi_i(s,\xi) = a_0 s + a_i$, for some constants $a_0,a_1,a_2\in\R$ with $a_0 \neq 0$, and 
		\[
		u = u_{c_1,H} \left( \frac{\tau - a_0 s - a_1}{\sqrt{1+a_0^2}} \right), \qquad \text{with $c_1 < 0$ if $H \le 0$.}
		\]		
	\end{itemize}
	\end{itemize}
Furthermore, if $\varphi_1$ and $\varphi_2$ are a-priori globally bounded, then conclusion (ii) holds with the second in \eqref{ipo_epigraph} replaced by the weaker 
	\begin{equation} \label{ipo_epi_weak}
	\limsup_{r \to \infty} \frac{|B_r^N|}{r^2 \log r} < \infty.
	\end{equation}
In this case, $\varphi_1$ and $\varphi_2$ are constant.
\end{theorem}
\subsubsection*{Capillary graphs as overdetermined problems}

The problem of classifying domains $\Omega$ supporting a solution of an overdetermined problem has a long history, starting from the semilinear case 
	\begin{equation}\label{eq_overdet_lapla}
	\left\{ \begin{array}{l}
	\Delta u + f(u) = 0 \qquad \text{on } \, \Omega \\[0.2cm]
	u > 0 \qquad \text{on } \, \Omega, \\[0.2cm]
	u = 0, \ \ \partial_{\bar \eta} u = \mathrm{const} \qquad \text{on } \, \partial \Omega
	\end{array}\right.
	\end{equation}
with $f \in \lip_\loc(\R)$. Connected, open sets $\Omega$ with smooth boundary supporting a \emph{bounded}, non-constant solution of \eqref{eq_overdet_lapla} are called $f$-extremal domains. While J.Serrin in \cite{serrin_movingplane} showed that the only bounded $f$-extremal domain in $\R^m$ is the round ball (cf. also \cite{weinberger} for $f \equiv 1$), in the past 25 years a major open problem in the field was to characterize unbounded $f$-extremal domains. In \cite{bcn_1}, H.Berestycki, L.Caffarelli and L.Nirenberg conjectured that the only $f$-extremal domains in $\R^m$ with connected complement are either the ball, the half-space, the cylinder $B^k \times \R^{m-k}$ or their complements. The conjecture turns out to be false in full generality, by counterexamples in \cite{sicbaldi} ($m \ge 3$) and \cite{ros_ruiz_sic_2} ($m = 2$ and $\Omega$ an exterior region). Surprisingly, in $\R^2$ the conjecture is true if $\partial \Omega$ is unbounded and connected \cite{ros_ruiz_sic_1}. \par
	It is natural to wonder whether $f$-extremal domains can be classified in more general Riemannian manifolds. The BCN conjecture was considered in the hyperbolic plane $\mathbb{H}^2$ and in the sphere $\Sph^2$, respectively in \cite{espi_far_maz} and \cite{espi_maz}, with different and interesting techniques. From a perspective more closely related to our work, the classification of solutions of \eqref{eq_overdet_lapla} on manifolds with $\Ricc \ge 0$ was studied in \cite{fmv}. Rigidity, in this case, means both that the domain splits as a product $N \times [0, \infty)$, and that the solution only depends on the split half-line. \par 
	To the best of our knowledge, generalizations of \eqref{eq_overdet_lapla} to nonlinear operators have mainly focused on the $p$-Laplace equation $\Delta_p u + f(u) = 0$, with 
	\[
	\Delta_p u = \diver \left( |D u|^{p-2}Du\right) \qquad \text{with } \, p > 1,
	\]
leaving problems like \eqref{overdet} mostly unexplored. Taking into account that \eqref{eq_overdet} alone imposes restrictions on the geometry of $\Omega$ even without overdetermined boundary conditions (cf. \cite{lopez,collin_krust,imperapigolasetti}, and the survey in \cite{jfwang}), we could expect more rigidity than in the case of the Laplacian. However, as far as we know there is still no attempt to study the equivalent of the BCN conjecture for \eqref{overdet}. In the present  paper, we move some steps in this direction for constant $f$.

\subsubsection*{Assumptions and main results}
Our first requirement  
	\begin{equation}\label{infu}
	\inf_\Omega u > -\infty, 
	\end{equation}
instead of the stronger $u \in L^\infty(\Omega)$, is made necessary to include the relevant standard example of the half-hyperplane. We tried to keep our curvature requirement on $M$ to a minimum, in particular, avoiding to bound the sectional curvature of $M$. This is coherent with the above mentioned splitting theorems in \cite{cheeger_gromoll,liwang2,liwang,fmv}, all based on B\"ochner formulas and thus naturally related to Ricci lower bounds. Precisely, we assume
	\begin{equation}\label{eq_curvat}
	\left\{ \begin{array}{ll}
	\Ricc \ge - \kappa(1 + r^2) & \quad \text{on $M$, for some $\kappa \in \R^+$,} \\[0.2cm]
	\Ricc \ge 0 & \quad \text{on } \, \Omega, 
	\end{array}\right.
	\end{equation}
where $r$ is the distance from a fixed origin, and the inequalities are meant in the sense of quadratic forms. The first condition is technical, and might be removable. On the other hand, condition $\Ricc \ge 0$ on $\Omega$ is fundamental at various stages of the proof. We anticipate that it would be very interesting to obtain a splitting theorem in the spirit of Theorem \ref{teo_splitting} below on manifolds satisfying
	\[
	\Ricc \ge - (m-1) \kappa^2, 
	\]
with $\kappa>0$ constant. To this aim, the techniques developed in Li-Wang's \cite{liwang2,liwang} should be quite helpful.\par
	In the generality of \eqref{infu}, \eqref{eq_curvat}, a major issue is to obtain a gradient bound 
	\[
	\sup_\Omega |Du| < \infty.
	\]
The main source of difficulty is that the linearization of the mean curvature operator possesses two eigenvalues that behave quite differently as $|Du| \ra \infty$, while an assumption on $\Ricc$ just controls, via comparison theory, the full trace of the Hessian of the distance function $r$. Therefore, $r$ cannot be used in Korevaar's localization method, customarily exploited to get gradient estimates for the mean curvature equation under sectional curvature bounds (cf. \cite{cmmr,bcmmpr,rosenbergschulzespruck,spruck}, and the references therein). We shall overcome this problem by using, in place of $r$, exhaustion functions coming from potential theory, in particular, we use a duality principle recently discovered in \cite{marivaltorta,maripessoa,maripessoa_2}, called the AK-duality (Ahlfors-Khasminskii duality). We comment on this point later, and we suggest to consult \cite[Sec. 3]{bcmmpr} and \cite{cmmr} for more details. The global gradient estimate in Theorem \ref{intro-prop-bound} is the second main achievement of our work, inspired by one for minimal graphs in $\R \times M$ that we recently obtained in \cite{cmmr} to prove Bernstein and half-space properties on manifolds with Ricci lower bounds, see Section \ref{sec:grad-bound} below. \par
	Leaving aside the characterization of cylinders and complements of balls in $\R^m$, that relate to the position vector field, an important step to show the uniqueness of the half-space among a large class of non-compact, theoretically $f$-extremal domains $\Omega \subset \R^m$, is to prove that $u$ is monotone in one variable on $\Omega$, namely that $\partial_m u > 0$. This is generally a hard task, and heavily depends on the behaviour of  $f$ and on $\Omega$. For instance, $\partial_m u > 0$ follows if $f$ is the derivative of a bistable nonlinearity, like in the Allen-Cahn equation 
	\[
	\Delta u + u - u^3 = 0,
	\]
and if $\Omega$ is a globally Lipschitz epigraph in the $x_m$-direction, see \cite[Thm. 1.1]{bcn_1}. Note that, since $\partial_m u$ satisfies the linearized equation $\Delta w + f'(u)w = 0$, its positivity implies that $u$ is a stable solution. The existence of the parallel field $\partial_m$ in $\R^m$ is also automatic in the setting of Theorem \ref{thm-epigraph} above, but on more general manifolds will be weakened to the assumption that $\Omega$ supports a bounded Killing field $X$. Next, a necessary condition to obtain the monotonicity $(Du,X)>0$ is given by some form of trasversality of $X$ to $\partial \Omega$, that is coherent with the sign of $\partial_{\bar \eta} u$ on each connected component of $\partial \Omega$. This last requirement, codified by \eqref{ipo_X2} below, is satisfied for instance if $\Omega$ is a locally Lipschitz epigraph. For constant $f$, the method in \cite{bcn_1} to prove $\partial_m u > 0$ cannot be applied, so we need to devise a different strategy, of independent interest. To reach the goal, we shall require that $(\overline{\Omega},\sigma)$ is a parabolic manifold with boundary, according to the following

\begin{definition}
Let $N$ be a smooth manifold with, possibly, non-empty $C^1$-boundary. Then, $N$ is said to be parabolic if the capacity of every compact set $K \subset N$, defined as  
	$$
	\capac(K) : = \inf\left\{ \int_N |D\phi|^2 \di x : \phi\in \lip_c(N), \, \phi \geq 1 \, \text{ on } \, K \right\},
	$$
vanishes.  
\end{definition}

If $\partial N \neq \emptyset$, note that the support of $\phi$ can intersect $\partial N$. It can be shown that $N$ is parabolic if and only if the Brownian motion on $N$, normally reflected on $\partial N$ if $\partial N \neq \emptyset$, is recurrent, cf. \cite{grigoryan}. 

\begin{example}
\emph{From the very definitions, if $M$ is a parabolic manifold without boundary, cf. \cite{grigoryan}, then every open subset $\Omega \subset M$ with $C^1$-boundary has the property that $\overline{\Omega}$ is parabolic (hereafter, we say that $\Omega$ \emph{has a  parabolic closure}). Notice also that, in dimension $2$, the parabolicity of $\overline\Omega$ is invariant by conformal deformations of the metric. By \cite{imperapigolasetti}, a sufficient condition for the parabolicity of $\overline{\Omega}$ is the validity of
$$
	\int^{\infty} \frac{\di s}{|\Omega \cap \partial B_s|} = \infty
$$
for balls $B_s$ centered at some fixed origin $o\in M$ (they assume a smooth boundary, but $C^1$ regularity is sufficient for their argument). An application of H\"older inequality (cf. \cite[Prop. 1.3]{rigolisetti}), shows that the condition is implied by 
$$
	\int^{\infty} \frac{s \di s}{|\Omega \cap B_s|} = \infty,
$$
that is satisfied, for instance, if the volume of geodesic balls in $\Omega$ grow at most quadratically. Examples include any smooth domain in the Euclidean plane $\R^2$, as well as any smooth domain contained between two parallel planes in $\R^3$. Also, by Bishop-Gromov comparison theorem, any smooth domain in a surface with non-negative sectional curvature has a parabolic closure.  
}
\end{example}

We are ready to state our main result:

\begin{theorem}\label{teo_splitting}
	Let $(M^m, \sigma)$ be a complete Riemannian manifold of dimension $m \ge 2$, and let $\Omega \subseteq M$ be a connected open set with smooth boundary such that $\overline\Omega$ is parabolic. Assume that 
	\[
	\left\{ \begin{array}{ll}
	\Ricc \ge - \kappa(1 + r^2) & \quad \text{on } \, M, \\[0.2cm]
	\Ricc \ge 0 & \quad \text{on } \, \Omega, 
	\end{array}\right.
	\]
for some constant $\kappa>0$, where $r$ is the distance from a fixed origin. Split $\partial \Omega$ into its connected components $\{\partial_j\Omega\}$, $1 \le j \le j_0$, possibly with $j_0 = \infty$. Let $u \in C^2(\overline\Omega)$ be a non-constant solution of the capillarity problem
	\begin{equation} \label{capill-eq}
		\begin{cases}
			\diver \left( \dfrac{Du}{\sqrt{1+|Du|^2}} \right) = H & \qquad \text{on } \, \Omega, \\
			u = b_j, \ \partial_{\bar\eta} u = c_j & \qquad \text{on } \, \partial_j \Omega, \; 1 \leq j \leq j_0, \\
			\inf_\Omega u > -\infty,
		\end{cases}
	\end{equation}
where $\bar \eta$ is the unit exterior normal to $\partial \Omega \hookrightarrow \Omega$, $H, c_j, b_j \in \R$, $\{c_j\}$ is a bounded sequence if $j_0 = \infty$, and with the agreement $b_1 \le b_2 \le \ldots \le b_{j_0}$. Assume that either
	\begin{equation} \label{intro-dOm-hp}
		u \, \text{ is constant on } \, \partial\Omega \qquad \text{or} \qquad \liminf_{r\to \infty} \frac{\log|\partial\Omega \cap B_r|}{r^2} < \infty.
	\end{equation}
If there exists a Killing vector field $X$ on $\overline\Omega$ with the following properties:
	\begin{equation}\label{ipo_X1}	
		\sup_\Omega |X| < \infty,
	\end{equation}
	\begin{equation}\label{ipo_X2}
	    \left\{ \begin{array}{l}
		c_j (X, \bar \eta) \ge 0 \quad \text{on $\partial_j \Omega$, for every $j$,} \\[0.2cm]
		c_j (X, \bar \eta) \not \equiv 0 \quad \text{on $\partial_j \Omega$, for some $j$}
	\end{array}\right.
	\end{equation}
then: 
	\begin{itemize}
		\item[(i)] $\Omega = (0, T) \times N$ with the product metric, for some $T \le \infty$ and some complete, boundaryless, parabolic $N$ with $\Ricc_N \ge 0$,
		\item[(ii)] the product $(X, \partial_t)$ is a positive constant, and $(Du,X) > 0$ on $\Omega$,
		\item[(iii)] $c_1 \le 0$ ($c_1 < 0$ if $H \le 0$) and, denoting with $\partial_1 \Omega = \{0\} \times N$, $u(t,x) = b_1 + u_{c_1,H}(t)$ is a (translated) standard example, with $u_{c_1,H}(t)$ as in \eqref{eq_split}.
	\end{itemize}
\end{theorem}

\begin{remark}
	\emph{The splitting of $\Omega$ as in the conclusion of Theorem \ref{teo_splitting} forces $\partial\Omega$ to consist of one or two copies of a connected, complete manifold of non-negative Ricci curvature. Hence, the second condition in \eqref{intro-dOm-hp} is satisfied a posteriori. In fact, by Bishop's theorem we have $|\partial\Omega \cap B_r(o)| \leq C_0 r^{m-1}$ for some $C_0 > 0$.
	}
\end{remark}


\begin{remark}
\emph{Condition \eqref{ipo_X2} is a mild transversality assumption, that can be rephrased as $(Du,X) \ge 0$ and $\not \equiv 0$ on $\partial \Omega$. Clearly, it is satisfied a posteriori if $\Omega$ splits as indicated in the above theorem. 
}
\end{remark}

\begin{remark}
	\emph{Theorem \ref{teo_splitting} is also related to the celebrated Schiffer's conjecture. The latter asks whether a domain $\Omega$ supporting a non-constant solution of
		\[
		\left\{ \begin{array}{l}
		\Delta u + \lambda u = 0 \qquad \text{on } \, \Omega, \\[0.2cm]
		u = b, \ \ \  \partial_{\bar \eta} u = c \qquad \text{on } \, \partial \Omega,
		\end{array}\right.
		\]
		for some constants $\lambda, b,c \in \R$, is necessarily a ball. The problem in unbounded domains has been considered in \cite{far_vald_ARMA}.
	}
\end{remark}

The proof of Theorem \ref{teo_splitting} hinges on two main results, of independent interest. The first is a geometric Poincar\'e inequality. To state it, let us fix some notation. For $u \in C^2(\Omega)$, we denote by $(\Sigma,g)$ its graph over $\Omega$
$$
	\Sigma = \{ (u(x),x) : x \in \Omega \} \subseteq \R\times M
$$
endowed with the metric $g$ induced from the ambient product metric $\di y^2 + \sigma$ on $\R \times M$, with $y$ the canonical coordinate on $\R$. We let $\nabla$, $\|\;\|$, $\di x_g$ be the Levi-Civita connection, vector norm and volume measure induced by $g$. For any point $x\in\Sigma$ where $\di u \neq 0$, the level set $\{u=u(x)\}$ is a regular embedded hypersurface of $\Sigma$ in a suitable neighbourhood of $x$. We let $A$ be its second fundamental form in $(\Sigma,g)$, and for any $v : \Sigma \to \R$ we let
$$
	\nabla_\top v := \nabla v - \left\langle \nabla v,\frac{\nabla u}{\|\nabla u\|} \right\rangle \frac{\nabla u}{\|\nabla u\|}
$$
be the component of $\nabla v$ tangent to $\{u=u(x)\}$. Then, along $\{u=u(x)\}$ the remainder in the classical Kato inequality is made explicit by the following identity from \cite{stezum}:
	\begin{equation}\label{eq_stezum}
	\|\nabla^2 u\|^2 - \|\nabla\|\nabla u\|\|^2 = \|\nabla_\top\|\nabla u\|\|^2 + \|\nabla u\|^2 \|A\|^2.
	\end{equation}
Note that $\|\nabla u\|$ is $C^1$ in the set $\{\di u \neq 0\}$.

\begin{proposition}
	Let $(M,\sigma)$ be a complete Riemannian manifold, $\Omega\subseteq M$ a smooth, connected open subset and $u\in C^2(\overline\Omega)$ satisfy \eqref{overdet} for some constant $H\in\R$. Assume that $u$ is strictly monotone in the direction of a Killing field $X$ on $\overline\Omega$, that is, $\bar v : = (Du,X)>0$ on $\Omega$. Then, for every $\varphi\in\lip_c(\overline\Omega)$
	\begin{equation}\label{intro-poincare}
	\begin{array}{l}
		\disp \int_\Sigma \left[ W^2 \left( \|\nabla_\top \|\nabla u\|\|^2 + \|\nabla u\|^2 \|A\|^2 \right) + \frac{\Ricc(Du,Du)}{W^2} \right] \varphi^2 \di x_g + \\[0.5cm]
		\qquad \qquad \disp + \int_\Sigma \frac{\bar v^2}{W^2} \left\| \nabla \left( \frac{\varphi\|\nabla u\|W}{\bar v}\right) \right\|^2\di x_g \le \int_\Sigma \|\nabla u\|^2 \|\nabla \varphi \|^2 \di x_g,
	\end{array}
	\end{equation}
	where $W = \sqrt{1+|Du|^2} \equiv 1/\sqrt{1-\|\nabla u\|^2}$.
\end{proposition}

\begin{remark}
\emph{The key point is that the support of the function $\varphi$ is allowed to meet $\partial\Omega$. Indeed, the overdetermined conditions force the boundary terms to cancel out. 
}
\end{remark}

Geometric Poincar\'e formulas of the type in \eqref{intro-poincare} are not new for boundaryless manifolds. For instance, R.Schoen and S.T.Yau \cite{schoenyau} and later P.Li and J.Wang \cite{liwang_mini} used an inequality similar to \eqref{intro-poincare} to prove the rigidity of certain minimal hypersurfaces properly immersed into manifolds with non-negative sectional curvature. Related inequalities also appeared in \cite{liwang2,liwang}, to split manifolds with $\Ricc \ge -(m-1)\kappa^2$ and whose Laplacian has a sufficiently large first eigenvalue.\par 
From a different point of view, formulas like \eqref{intro-poincare} were independently introduced by A.Farina in his habilitation thesis \cite{far_HDR} and by Farina, B.Sciunzi and E.Valdinoci in \cite{fsv}, to study Gibbons and De Giorgi type conjectures for stable solutions of quasilinear equations of the type
	\[
	\diver \left( \frac{\varphi(|Du|)}{|Du|}Du \right) + f(u) = 0 \qquad \text{on } \, \R^m,
	\]
with rather general $\varphi$ (cf. \cite[Thm. 2.5]{fsv}). Their goal is to prove the 1D-symmetry of $u$, that is, that $u$ only depends on one variable up to rotation. For the mean curvature operator $\varphi(t) = t/\sqrt{1+t^2}$, the inequality reads
	\begin{equation}\label{eq_farina}
	\int_{\R^m} \left[ \frac{|D_\top|Du||^2}{(1+|Du|^2)^{3/2}} + \frac{|A|^2|Du|^2}{\sqrt{1+|Du|^2}} \right]\varphi^2 \le \int_{\R^m} \frac{|Du|^2|D\varphi|^2}{\sqrt{1+|Du|^2}} \qquad \forall \, \varphi \in \lip_c(\R^m),
	\end{equation}
where, at a point $x$ such that $Du(x) \neq 0$, $D_\top$ and $A$ are the gradient and second fundamental form of the level set $\{u = u(x)\}$ in $\R^m$. We refer the reader, in particular, to Theorems 1.1, 1.2 and 1.4 in \cite{fsv}. 

Regarding overdetermined boundary value problems, Farina and Valdinoci in \cite{far_vald_ARMA} made the remarkable discovery that the identity corresponding to \eqref{eq_farina} for $\Delta u + f(u) = 0$, meaning 
	\begin{equation}\label{eq_farina_2}
	\int_{\Omega} \left[ |D_\top|Du||^2 + |A|^2|Du|^2 \right]\varphi^2 \le \int_{\Omega} |Du|^2|D\varphi|^2,
	\end{equation}
still holds even if the support of $\varphi$ contains a portion of $\partial\Omega$, provided that $u$ satisfies \eqref{eq_overdet_lapla} and $\partial_m u > 0$. In other words, the contributions of boundary terms to \eqref{eq_farina_2} on $\partial \Omega \cap {\rm spt } \varphi$ vanish identically if both $u$ and $\partial_{\bar \eta} u$ are constant on $\partial \Omega$. This opened the way to use \eqref{eq_farina_2} to characterize domains supporting a non-constant solution of \eqref{eq_overdet_lapla}, a point of view further developed in \cite{fmv} to obtain splitting theorems for domains with $\Ricc \ge 0$ provided that the solution of \eqref{eq_overdet_lapla} is monotone in the direction of a Killing field. In the nonlocal setting, Y.Sire and Valdinoci in \cite{sireval} used a similar technique on product domains $M\times[0,\infty)$, where $M$ is a complete manifold without boundary and with $\Ricc\geq 0$, to prove rigidity properties for solutions of Neumann problems involving a weighted (local) Laplacian, that are related to nonlocal equations for the fractional Laplacian on $M$ via the well-known extension argument. In our paper, a key point to get \eqref{intro-poincare} is to show that the ``magic cancelation" in \cite{far_vald_ARMA} still holds for capillary graphs. We give a simpler proof of this identity, emphasizing its geometrical meaning.\\
\par
The second result is a gradient estimate for non-negative solutions of the CMC equation, holding under just a lower bound on the Ricci curvature of $(M,\sigma)$ in $\Omega$ and, in some cases, a mild control on the volume growth of $\partial\Omega$. This improves on \cite{cmmr}, that considers the case $H=0$.

\begin{theorem} \label{intro-prop-bound}
	Let $(M,\sigma)$ be a complete Riemannian manifold of dimension $m\geq 2$, and let $\Omega \subset M$ be an open subset. Assume that
	\begin{equation}\label{eq_Ricci}
		\left\{ \begin{array}{ll}
		\Ricc \ge - K(1+r^2) & \quad \text{on } \, M, \\[0.2cm]
		\Ricc \geq -(m-1)\kappa^2 & \quad \text{on } \, \Omega,
		\end{array}\right.
	\end{equation}
for some constants $K>0$, $\kappa\ge0$, and where $r$ is the distance from a fixed origin $o \in M$. Let $u\in C^2(\Omega)$ satisfy
	$$
		\diver \left( \dfrac{Du}{\sqrt{1+|Du|^2}} \right) = H, \qquad u \ge 0 \qquad \text{on } \, \Omega
	$$
	for some constant $H\in\R$.
	%
	%
	Assume that at least one of the following conditions is satisfied
	\begin{itemize}
		\item [-] $\Omega = M$
		\item [-] $u\in C(\overline{\Omega})$ and $u|_{\partial\Omega}$ is constant
		\item [-] $\Omega$ has locally Lipschitz boundary and
		$$
			\liminf_{r\to \infty} \frac{\log|\partial\Omega \cap B_r(o)|}{r^2} < \infty.
		$$
	\end{itemize}
	Let $C\geq 0$ be such that
	\begin{equation}\label{eq_CH}
	\begin{array}{ll}
	\disp C^2 \ge (m-1)\kappa^2 - \frac{H^2}{m} & \qquad \text{if } \, H \leq 0, \\[0.4cm]
	\disp C^2 > (m-1)\kappa^2 - \frac{H^2}{m} & \qquad \text{if } \, H > 0, 
	\end{array}
	\end{equation}
and choose $A \ge 1$ to satisfy 	
	\[
	\frac{H^2}{m} - \frac{CH}{t} + \left( C^2 - (m-1)\kappa^2 \right) \frac{t^2-1}{t^2} \geq 0 \qquad \text{for every } \, t \geq A.
	\]	
	Then
	\begin{equation} \label{intro-bound}
		\sup_\Omega \frac{\sqrt{1+|Du|^2}}{e^{Cu}} \leq \max \left\{ A, \limsup_{x\to\partial\Omega} \frac{\sqrt{1+|Du(x)|^2}}{e^{Cu(x)}} \right\},
	\end{equation}
	and in case $\Omega = M$ this yields
	$$
		\sqrt{1+|Du|^2} \leq A e^{Cu} \qquad \text{in } \, M.
	$$
\end{theorem}



\begin{remark}
\emph{If $\Omega \subset M$ is open with non-empty boundary, and $v : \Omega \ra \R$, we agree to write 
	\[
	\limsup_{x \to \partial \Omega} v \doteq \inf_{V \text{ open, } \overline{V} \subset \Omega} \left( \sup_{\Omega \backslash \overline{V}} v \right),
	\]
where $\overline{V}$ is the closure of $V$ in $M$. 
}
\end{remark}

A systematic discussion of values of $A$, $C$ satisfying the requirements of Theorem \ref{intro-prop-bound} for various $\kappa\geq0$, $H\in\R$ is carried out in Section \ref{sec:grad-bound}. 
In particular, for $H\leq 0$ we can choose $A=1$, $C=\sqrt{m-1}\kappa$ and we recover the gradient bound for positive minimal graphic functions obtained in \cite{cmmr},
$$
	\sup_\Omega \frac{\sqrt{1+|Du|^2}}{e^{\sqrt{m-1}\kappa u}} \leq \max \left\{ 1, \limsup_{x\to\partial\Omega} \frac{\sqrt{1+|Du(x)|^2}}{e^{\sqrt{m-1}\kappa u(x)}} \right\}.
$$
When $\kappa=0$, the choice $A=1$, $C=0$ is admissible for each $H\in\R$, so for any CMC-graphic function $u\geq 0$ from \eqref{intro-bound} we have
\begin{alignat*}{2}
	\sup_\Omega |Du| & = \limsup_{x\to\partial\Omega} |Du|(x) && \qquad \text{if } \, \Omega \neq M, \\
	Du & \equiv 0 && \qquad \text{if } \, \Omega = M.
\end{alignat*}

\begin{remark}
\emph{Although the conclusion of Theorem \ref{intro-prop-bound} is localized on an open subset $\Omega \subset M$, a global curvature bound on $M$ like the first in \eqref{eq_Ricci} is crucial in a technical step of the proof. The existence of extensions $M$ of $\Omega$ satisfying the first in \eqref{eq_Ricci} is not automatic: indeed, an example constructed in \cite{pigovero} shows that a smooth, complete manifold with boundary $\overline{\Omega}$ satisfying $\Ricc \ge - \kappa^2$ ($\kappa \in \R$) may not be isometrically embeddable into a complete manifold without boundary $(M,\sigma)$ of the same dimension, if we require that the Ricci curvature of $M$ is bounded from below by some other, possibly different, constant. Arguing similarly, one may be able to find an example of $\overline{\Omega}$ with $\Ricc \ge - \kappa^2$ that is not embeddable into a complete, boundaryless manifold $(M,\sigma)$ of the same dimension whose Ricci curvature satisfies the first in \eqref{eq_Ricci}. 
}
\end{remark}

\begin{remark}
\emph{The conditions in \eqref{eq_CH} relate to classical obstructions for the existence of entire CMC graphs over $M$ (with no a priori bound). Indeed, extending previous results by E.Heinz \cite{heinz}, S.-S.Chern \cite{chern} and H.Flanders \cite{flanders} in the Euclidean setting, I.Salavessa \cite{salavessa} showed that a complete manifold $M$ with 
	\[
	\Ricc \ge - (m-1)\kappa^2
	\]
does not support any entire solution of 
\begin{equation} \label{CMCeq-salavessa}
	\diver \left( \frac{Du}{\sqrt{1+|Du|^2}} \right) = H
\end{equation}
whenever $|H|>(m-1)\kappa$, while entire non-negative solutions of \eqref{CMCeq-salavessa} exist on the $m$-dimensional hyperbolic space with sectional curvature $-\kappa$ for each value of $0 \leq H \leq (m-1)\kappa$. 
}
\end{remark}

\section{Preliminaries}\label{sec_prelim}

\subsection{CMC graphs}

Let $(M, \sigma)$ be a Riemannian manifold of dimension $m$, with volume measure $\di x$ and induced $(m-1)$-dimensional Hausdorff measure $\di \mathscr{H}^{m-1}$. The metric $\sigma$ will also be denoted with $( \, , \, )$. We let $| \cdot |$ and $D$ denote, respectively, the norm and Levi-Civita connection of $\sigma$. Give coordinates $(y,x)$ on $\bar M = \R \times M$, and let $\bar D$ denote the Levi-Civita connection of the product metric $\metric = \di y^2 + \sigma$ on $\bar M$.

Let $\Omega\subseteq M$ be an open subset and $u : \Omega \to \R$ a twice differentiable function. We let
$$
	\Sigma = \{ (u(x),x) : x \in \Omega \} \subseteq \R\times M
$$
be the graph of $u$ over $\Omega$, and we denote by $g$ the graph metric induced on $\Sigma$ from $\metric$, and by $\|\cdot\|$, $\nabla$, $\Delta_g$ the induced norm, connection and Laplace operator on $(\Sigma,g)$.

The graph map
$$
	f : \Omega \ra \bar M : x \mapsto (u(x),x)
$$
is a diffeomorphism onto the image $f(\Omega) = \Sigma$, whose inverse $\pi : \Sigma \to \Omega$ is the restriction to $\Sigma$ of the canonical projection $\R \times M \to M$. We use these maps to identify the graph $\Sigma$ with the base domain $\Omega$. In particular, the pulled-back metric $f^\ast g$ on $\Omega$ will still be denoted by $g$, and $\|\cdot\|$, $\nabla$, $\Delta_g$ will also denote the norm, connection and Laplace operator of the resulting manifold $(\Omega,g)$.


Let $\{\partial_j\}$ be a local coordinate frame on $(M,\sigma)$, with $\sigma = \sigma_{ij} \di x^i \otimes \di x^j$. We write the graph metric $g$ on $\Sigma$ as 
$$
	g_{ij} = \sigma_{ij} + u_iu_j,
$$
where $\di u = u_i \di x^i$. Defining $u^j = \sigma^{jk}u_k$ and $W^2 = 1+ |Du|^2 = 1+ u_i u^i$, the components of the inverse $g^{ij}$ are 
$$
	g^{ij} = \sigma^{ij} - \frac{u^iu^j}{W^2},
$$
and the Riemannian volume and hypersurface measures of $g$ write as
\[ 
	\di x_g = W \di x, \qquad \di \mathscr{H}^{m-1}_g = W \di \mathscr{H}^{m-1}.
\]
For any function $\phi \in C^1(\Omega)$, we write $\di \phi = \phi_i \di x^i$ and $D\phi = \phi^i \partial_i$ with $\phi^i = \sigma^{ik}\phi_k$. Then, $\nabla\phi$ is given in local components by
$$
	\nabla\phi = g^{ik}\phi_k\partial_i = D\phi - \frac{(D\phi,D u)}{W^2} Du
$$
In particular, note that 
\begin{equation}\label{nablau2}
	\nabla u = \frac{Du}{W^2}, \qquad \|\nabla u\|^2 = g^{ij} u_i u_j = \frac{W^2-1}{W^2}, \qquad W^{-2} = 1- \|\nabla u\|^2
\end{equation}
and that, for every $\phi\in C^1(\Omega)$,
\begin{equation} \label{D_phi_nabla_phi}
    \langle \nabla\phi,\nabla u\rangle = g^{ij}\phi_i u_j = \frac{(D\phi,Du)}{W^2}.
\end{equation}
%
%

With the agreements of the Introduction, the normal vectors to $\Sigma$ and $M$, pointing outward from the subgraph of $u$, are respectively
\begin{equation}\label{normals}
{\bf n} = \frac{\partial_y - u^j \partial_j}{W}, \qquad {\bf \bar n} = -\partial_y.
\end{equation}
Thus, 
$$
\langle \eta, \bar \eta \rangle = - \langle {\bf n}, {\bf \bar n} \rangle = \frac{1}{W},
$$
and the angle condition $\langle\eta, \bar \eta \rangle = \cos \gamma$ rewrites as $|D u| = \tan \gamma$. Requiring $u=0$ on $\partial \Omega$ yields $\partial_{\bar \eta} u = -|D u|$ and we deduce \eqref{eq_overdet}.

A differentiation gives that the second fundamental form $\II_\Sigma$ and the unnormalized mean curvature $H$ in the ${\bf n}$ direction have components 
\begin{equation}\label{seconffund}
	\mathrm{II}_{ij} = \frac{u_{ij}}{W}, \qquad H = g^{ij}\mathrm{II}_{ij} = \diver \left( \frac{Du}{W}\right),
\end{equation}
with $\diver$ the divergence in $\sigma$. Moreover, from the identity
\[
	\Gamma_{ij}^k = \gamma^k_{ij} - \frac{u^k u_{ij}}{W^2}
\]
relating the Christoffel symbols $\Gamma_{ij}^k$ and $\gamma_{ij}^k$ of, respectively, $g$ and $\sigma$, for every $\phi : M \to \R$ the components of the graph Hessian $\nabla^2 \phi$ and of $\Delta_g \phi$ can be written as
\begin{equation}\label{hessian_on_graph}
	\left\{ \begin{array}{l}
		\nabla^2_{ij} \phi = \phi_{ij} - \phi_ku^k \dfrac{u_{ij}}{W^2} \\[0.2cm]
		\Delta_g \phi = g^{ij}\phi_{ij} - \phi_ku^k \dfrac{H}{W}.
	\end{array}\right.
\end{equation}
In particular, 
\begin{equation}\label{hessian_on_graph_u}
	\nabla^2_{ij} u = \frac{u_{ij}}{W^2} = \frac{\II_{ij}}{W}, \qquad \Delta_g u = \frac{H}{W}.
\end{equation}


From now on we assume that $H$ is constant. For every Killing field $\bar X$ defined in a neighbourhood $U\subseteq\bar M$ of the graph $\Sigma$, the angle function $\Theta_{\bar X} : = \langle \mathbf{n}, \bar X \rangle$ solves Jacobi equation
\begin{equation} \label{Jeq}
	0 = J_\Sigma \Theta_{\bar X} : = \Delta_g \Theta_{\bar X} + \Big( \|\II_\Sigma\|^2 + \overline{\Ricc}({\bf n},{\bf n})\Big) \Theta_{\bar X},
\end{equation}
with $\overline{\Ricc}$ the Ricci curvature of $\bar M$. This is the case, for instance, of the angle function $\Theta_{\partial_y} = \langle \mathbf{n},\partial_y \rangle = W^{-1}$ associated to the Killing field $\partial_y$. As a consequence $W$ satisfies
\begin{equation}\label{eq_W}
	\Delta_g W = \left(\|\II_\Sigma\|^2 + \overline{\Ricc}({\bf n}, {\bf n})\right) W + 2 \frac{\|\nabla W\|^2}{W}.
\end{equation}
If $X$ is a Killing field in $(\Omega,\sigma)$ then we can extend it by parallel transport on the cylinder $\R\times\Omega \subseteq \bar M$ to a Killing field $\bar X$ satisfying $\langle \partial_y,\bar X \rangle = 0$, with corresponding angle function $\Theta_{\bar X} = \langle \mathbf{n}, \bar X \rangle = W^{-1} (Du, X)$. Since \eqref{Jeq} holds for both $\Theta_{\partial_y}$ and $\Theta_{\bar X}$, the quotient 
	\[
	\bar v : = \Theta_{\bar X} / \Theta_{\partial_y} = (Du, X)
	\]
is a solution of
\begin{equation} \label{killing-Du}
	\Delta_g \bar v - 2\left\langle \frac{\nabla W}{W}, \bar v \right\rangle = 0.
\end{equation}
By introducing the operator
\begin{equation} \label{LLW-def}
	\LL_W \phi : = W^{2} {\rm div}_g \big( W^{-2} \nabla \phi \big)	= \Delta_g \phi - 2 \left\langle \frac{\nabla W}{W}, \nabla \phi \right\rangle,
\end{equation}
equations \eqref{eq_W} and \eqref{killing-Du} can be rewritten as
\begin{equation} \label{LLWvbar}
	\LL_W \bar v = 0
\end{equation}
and
\begin{equation} \label{LLWW}
	\LL_W W = \left(\|\II_\Sigma\|^2 + \overline{\Ricc}({\bf n}, {\bf n}) \right) W.
\end{equation}
Considering the weighted measures
\[
	\di x_W = W^{-2} \di x_g, \qquad \di \mathscr{H}^{m-1}_W = W^{-2} \di \mathscr{H}^{m-1}_g,
\]
%
%
%
%
we note that $\LL_W$ is symmetric with respect to $\di x_W$. For a given $C\in\R$, we compute
$$
	\Delta_g e^{-Cu} = -C e^{-Cu} \Delta_g u + C^2 e^{-Cu} \|\nabla u\|^2 = e^{-Cu} \left(C^2 \|\nabla u\|^2  - \frac{CH}{W}\right),
$$
so the function
$$
	z = \frac{W}{e^{Cu}}
$$
satisfies
\begin{equation} \label{LLWz}
\begin{split}
	\LL_W z & = W \Delta_g e^{-Cu} + e^{-Cu} \LL_W W \\
	& = \left( \|\II_\Sigma\|^2 - \frac{CH}{W} + \overline{\Ricc}({\bf n},{\bf n}) + C^2 \|\nabla u\|^2 \right) z.
\end{split}
\end{equation}
We observe that if
\begin{equation}\label{iporicci}
	\Ricc \ge -(m-1)\kappa^2, 
\end{equation}
for some $\kappa \ge 0$, then, using \eqref{nablau2},
\begin{equation}\label{eq_lowerricci_pre}
	\overline{\Ricc}({\bf n},{\bf n}) = \Ricc \left(\frac{Du}{W},\frac{Du}{W}\right) \ge - (m-1)\kappa^2 \frac{W^2-1}{W^2} = -(m-1)\kappa^2 \|\nabla u\|^2.
\end{equation}
By Cauchy inequality we also have
\begin{equation}\label{Newton}
	\|\II_\Sigma\|^2 \geq \frac{1}{m}\left( \mathrm{trace}_g(\II_\Sigma) \right)^2 = \frac{H^2}{m},
\end{equation}
and therefore, under assumption \eqref{iporicci}, equation \eqref{LLWz} gives
\begin{equation} \label{LLWz1}
	\LL_W z \geq \left( \frac{H^2}{m} - \frac{CH}{W} + (C^2 - (m-1)\kappa^2) \frac{W^2-1}{W^2} \right) z.
\end{equation}

We conclude this section by proving an analogous differential inequality for a modification of the function $z$. Let $\psi_0$ be a positive function satisfying
$$
	\Delta_g \psi_0 \leq \lambda\psi_0
$$
for some $\lambda \in \R^+$. Let $\beta>0$ be given and set
$$
	\psi = \psi_0^{-\beta}.
$$
Then,
\begin{align*}
	\frac{\nabla\psi}{\psi} & = - \beta \frac{\nabla\psi_0}{\psi_0}, \\
	\Delta_g \psi & = -\beta \psi_0^{-\beta-1} \Delta_g \psi_0 + \beta(1+\beta) \psi_0^{-\beta-2}\|\nabla\psi_0\|^2 \\
	& = \psi \left( - \beta\frac{\Delta_g \psi_0}{\psi_0} + \frac{1+\beta}{\beta} \frac{\|\nabla\psi\|^2}{\psi^2} \right) \\
	& \geq \psi \left( - \beta\lambda + \frac{1+\beta}{\beta} \frac{\|\nabla\psi\|^2}{\psi^2} \right).
\end{align*}
We now compute
\begin{align*}
	\Delta_g (\psi e^{-Cu}) & = e^{-Cu} \Delta_g \psi + \psi \Delta_g (e^{-Cu}) - 2 e^{-Cu} \langle C\nabla u,\nabla \psi \rangle \\
	& \geq \psi e^{-Cu}\left( - \beta\lambda + \frac{1+\beta}{\beta} \frac{\|\nabla\psi\|^2}{\psi^2} + C^2 \|\nabla u\|^2 - 2\left\langle C\nabla u, \frac{\nabla \psi}{\psi} \right\rangle - \frac{C H}{W}\right)
\end{align*}
that, by Young's inequality
$$
	- 2\left\langle C\nabla u, \frac{\nabla \psi}{\psi} \right\rangle \geq - \frac{1+\beta}{\beta} \frac{\|\nabla\psi\|^2}{\psi^2} - \frac{\beta}{1+\beta} C^2 \|\nabla u\|^2,
$$
leads to
$$
	\Delta_g (\psi e^{-Cu}) \geq \psi e^{-Cu} \left( \frac{C^2\|\nabla u\|^2}{1+\beta} - \beta\lambda - \frac{CH}{W} \right).
$$
For $\delta>0$, we let
$$
	\eta = \psi e^{-Cu} - \delta, \qquad \tilde{z} = W \eta.
$$
Then, a direct computation gives
\begin{equation} \label{LLWz+}
	\begin{split}
		\LL_W \tilde{z} & = W \Delta_g \eta + \eta \LL_W W \\
		& \geq \left( \|\II_\Sigma\|^2 + \overline{\Ricc}({\bf n},{\bf n}) + \left(1+\frac{\delta}{\eta}\right) \left( \frac{C^2 \|\nabla u\|^2}{1+\beta} - \beta\lambda - \frac{CH}{W} \right) \right) \tilde{z}
	\end{split}
\end{equation}
in the set $\{ \eta > 0 \}$. Using again \eqref{eq_lowerricci_pre} and \eqref{Newton}, and the definition of $\tilde{z}$, under assumption \eqref{iporicci}, we obtain
\begin{equation} \label{LLWz+1}
\begin{array}{lcl}
	\LL_W \tilde{z} & \geq & \disp \left( \frac{H^2}{m} - \frac{CH}{W} - (m-1)\kappa^2 \frac{W^2-1}{W^2} \right. \\[0.4cm]
	& & \disp + \left. \left(1+\frac{\delta}{\eta}\right) \left( \frac{C^2}{1+\beta}\frac{W^2-1}{W^2} - \beta\lambda \right) - \frac{CH\delta}{\tilde{z}} \right) \tilde{z} \qquad \text{on } \, \big\{ \eta > 0\big\}.
\end{array}
\end{equation}

\section{AK-duality and global gradient bounds for CMC graphs} \label{sec:grad-bound}

The goal of this section is to prove the following 

\begin{theorem} \label{thm-CMCbound}
	Let $(M,\sigma)$ be a complete Riemannian manifold of dimension $m \geq 2$, and let $\Omega \subseteq M$ an open set such that
	$$
		\Ricc \geq -(m-1)\kappa^2 \qquad \text{in } \, \Omega
	$$
	for some constant $\kappa \geq 0$. Suppose that $u \in C^2(\Omega)$ satisfies
	$$
		\diver \left( \dfrac{Du}{\sqrt{1+|Du|^2}} \right) = H, \quad u \geq 0 \qquad \text{in } \, \Omega
	$$
	for some $H\in\R$, and that for some (hence, any) $q\in\Sigma$
	\begin{equation} \label{vol-gr-Sigma}
		\liminf_{r\to \infty} \frac{\log|B^g_r(q)|_g}{r^2} < \infty,
	\end{equation}
where $(\Sigma,g)$ is the graph of $u$ over $\Omega$. Let $C\geq 0$ satisfy
	\begin{align}
		\label{appr-hp1}
		\frac{H^2}{m} + C^2 - (m-1)\kappa^2 \geq 0 & \qquad \text{if} \quad H \leq 0, 
		\\
		\label{appr-hp2}
		\frac{H^2}{m} + C^2 - (m-1)\kappa^2 > 0 & \qquad \text{if} \quad H > 0 
	\end{align}
	and let $A\geq 1$ be such that
	\begin{equation} \label{CAHge0}
		\frac{H^2}{m} - \frac{CH}{t} + \left( C^2 - (m-1)\kappa^2 \right) \frac{t^2-1}{t^2} \geq 0 \qquad \text{for every } \, t \geq A.
	\end{equation}
	%
	%
	%
	Then,
	\begin{equation} \label{CMC-gen-bound}
		\frac{W}{e^{Cu}} \leq \max \left\{ A, \limsup_{y\to\partial\Omega} \frac{W(y)}{e^{C u(y)}} \right\} \qquad \text{on } \, \Omega.
	\end{equation}
	In particular, in case $\Omega=M$ we have
	$$
		\sqrt{1+|Du|^2} \leq A e^{Cu} \qquad \text{on } \, M.
	$$
\end{theorem}

We postpone the proof to subsection \ref{subs-bound-proof}, while commenting here on sufficient conditions for the validity of \eqref{vol-gr-Sigma} and on admissible choices of the parameters $A$, $C$ satisfying the above requirements. As remarked in the Introduction, via a calibration argument \eqref{vol-gr-Sigma} is satisfied under mild assumptions on $\partial\Omega$ and on the Ricci curvature of $(M,\sigma)$, made precise by the next Proposition \ref{prop-stoch}, that we draw from Lemma 2 of \cite{cmmr}. In particular, \eqref{vol-gr-Sigma} holds in case $\Omega = M$ and we obtain the subsequent Corollary \ref{cor-entire}.

\begin{proposition}[\cite{cmmr}] \label{prop-stoch}
	Let $M$ be a complete manifold of dimension $m \geq 2$. Let $o\in M$ be a fixed origin, set $r(x) = \dist_g(o,x)$ for $x\in M$ and assume that
	\begin{equation} \label{ipo-Ric-quad}
	\Ricc \ge - \kappa(1+r^2) \qquad \text{on } \, M
	\end{equation}
	for some constant $\kappa > 0$. Let $\Omega \subseteq M$ be an open subset supporting a solution $u \in C^2(\Omega)$ of
	$$
		\diver\left( \frac{Du}{\sqrt{1+|Du|^2}} \right) = H \qquad \text{in } \, \Omega
	$$
	for some bounded function $H : \Omega \to \R$. Assume that at least one of the following conditions is satisfied:
	\begin{itemize}
		\item [-] $\Omega = M$;
		\item [-] $u \in C(\overline{\Omega})$ and $u$ is constant on $\partial\Omega$;
		\item [-] $\Omega$ has locally  Lipschitz boundary and
		$$
		\liminf_{r\to \infty} \frac{\log|\partial\Omega \cap B_r|}{r^2} < \infty.
		$$
	\end{itemize}
	Then, \eqref{vol-gr-Sigma} holds for each fixed $q\in\Sigma$.
\end{proposition}

From a direct application of Theorem \ref{thm-CMCbound} and Proposition \ref{prop-stoch}, we readily deduce Theorem \ref{intro-prop-bound} in the Introduction.

\begin{remark}
	\emph{Direct inspection of the proof of Lemma 2 in \cite{cmmr} shows that Proposition \ref{prop-stoch} remains true if boundedness of $H$ is replaced by the weaker assumption
		$$
		\limsup_{r\to \infty} \frac{\log\left( \disp\int_{B_r\cap\Omega} |H| \right)}{r^2} < \infty.
		$$
	}
\end{remark}

\begin{corollary} \label{cor-entire}
	Let $(M,\sigma)$ be a complete Riemannian manifold of dimension $m \geq 2$ satisfying $\Ricc \geq -(m-1)\kappa^2$ for some $\kappa \geq 0$ and let $u \in C^2(M)$ be a solution of
	$$
		\diver \left( \dfrac{Du}{\sqrt{1+|Du|^2}} \right) = H, \quad u \geq 0 \qquad \text{in } \, M
	$$
	for some constant $H\in\R$. If $C \geq 0$ and $A \geq 1$ satisfy requirements of Theorem \ref{thm-CMCbound} then
	$$
		\sqrt{1+|Du|^2} \leq Ae^{Cu} \qquad \text{in } \, M.
	$$
\end{corollary}

\begin{remark} \label{rmk-admissible}
	\emph{We point out some admissible choices for $A$ and $C$ satisfying the requirements of Theorem \ref{thm-CMCbound}:
	\begin{itemize}
		\item [-] if $\kappa = 0$ then $A=1$, $C=0$ is admissible for every $H\in\R$;
		\item [-] if $H = 0$ and $\kappa \geq 0$ then $A = 1$, $C = \sqrt{m-1}\kappa$ is admissible;
		\item [-] if $|H| > \sqrt{m(m-1)}\kappa > 0$ then $A=1$, $C=0$ is admissible;
		\item [-] the case $\kappa > 0$ and $|H| \leq \sqrt{m(m-1)}\kappa$ is covered, amongst others, by the admissible
		\begin{alignat*}{2}
		& A = 1, \quad C = \sqrt{(m-1)\kappa^2 - H^2/m} && \qquad \text{for } \, -\sqrt{m(m-1)}\kappa \leq H \leq 0, \\
		& A = 1 + \sqrt{\frac{H}{\sqrt{m-1}\kappa}}, \quad C = A\sqrt{m-1}\kappa && \qquad \text{for every } \, \kappa > 0, \, H \geq 0, \\
		& A = \sqrt{1+\frac{m}{3}}, \quad C = 2\sqrt{m-1}\kappa && \qquad \text{for every } \, \kappa \geq 0, \, H \in \R.
		\end{alignat*}
	\end{itemize}
In particular, under assumptions \eqref{appr-hp1}-\eqref{appr-hp2} it is always possible to find $A\geq1$ satisfying \eqref{CAHge0}. More general sufficient conditions on $A,C$ depending on $\kappa,H,m$ will be given in Proposition \ref{prop_AC} in the Appendix.
}
\end{remark}

\subsection{The proof} \label{subs-bound-proof}

The core of the argument depends on a novel way to localize estimates for the constant mean curvature equation, that we discovered in \cite{cmmr} and that we refine in the present paper. 

\begin{remark}\label{rem_notbad}
\emph{Hereafter, we say that a subset $U \subset N$ is smooth (or, it has smooth boundary) if its closure $\overline{U}$ in $N$ is a smooth manifold with boundary, and if, for each $y \in \partial U$, there exists a chart $(V,\varphi) \subset N$, with $V$ diffeomorphic to $\R^m$, such that $\varphi(y) = 0$ and $\varphi(\overline{U})$ is a closed half-space.
}
\end{remark}

For suitable smooth open subsets $\Sigma' \subset \Sigma$, we shall produce a boundaryless manifold $(N,h)$ of the same dimension with $\overline{\Sigma'} \hookrightarrow N$ isometrically, and an exhaustion $\psi_0: N \ra \R$ solving 
	\begin{equation}\label{def_psi0}
	\Delta_h \psi_0 \le \lambda \psi_0 \qquad \text{on } \, N,
	\end{equation}
for a suitable, positive constant $\lambda$, that we exploit in place of the distance function from a fixed origin of $M$. Recall that $\psi_0$ is said to be an exhaustion if sublevel sets of $\psi_0$ are relatively compact in $N$. In this case, we also write 
	\[
	\psi_0(x) \ra +\infty \qquad \text{as } \, x \to \infty.
	\]
Recall from \cite{grigoryan} that a manifold $(N,h)$ without boundary is said to be stochastically complete if the minimal Brownian motion $\mathscr{B}_t$ on $N$ is non-explosive, that is, if the trajectories of $\mathscr{B}_t$ have infinite lifetime almost surely. Various sufficient conditions were shown to be equivalent to the stochastic completeness of $N$, among them we emphasize the following two:
\begin{itemize}
\item[$(i)$] \cite{grigoryan,prsmemoirs} for some (equivalently, every) $\lambda> 0$, the only weak solution $0 \le u \in W^{1,2}_\loc(N)\cap L^\infty(N)$ of $\Delta_h u \ge \lambda u$ is $u \equiv 0$;
\item[$(ii)$] \cite{prsmemoirs} the \emph{weak maximum principle at infinity} for the Laplace-Beltrami operator $\Delta_h$ holds, that is, for every $\varphi \in C^2(N)$ satisfying $\sup_N \varphi < \infty$, 
$$
	\inf_{\{\varphi >\gamma\}} \Delta_h \varphi \leq 0
$$
for each $\gamma < \sup_N \varphi$. 
\end{itemize}
If $(N,h)$ is geodesically complete, then by \cite{prsmemoirs} the bound 
	\[ 
	\liminf_{r \to \infty} \frac{\log |B_r^h|_h}{r^2} < \infty
	\]
for the growth of the volume of geodesic balls centered at a fixed origin suffices to guarantee the stochastic completeness of $N$ (see also \cite[Thm. 9.1]{grigoryan} for a similar condition). Note that the condition holds if 
	\[
	|B_r^h|_h \le C_1 \exp \big\{ C_2 r^2 \big\} \qquad \text{for constants } \, C_1,C_2>0,
	\]
and is therefore a rather mild assumption. In \cite{marivaltorta}, the authors realized that each of $(i)$ and $(ii)$ is equivalent (and not only necessary, as in \cite[Cor. 6.6]{grigoryan}) to the existence, for some (every) $\lambda > 0$, of a continuous exhaustion function $\psi_0$ satisfying \eqref{def_psi0} in the weak sense. Such a function is named a Khasminskii type potential in \cite{marivaltorta,maripessoa,maripessoa_2}. The characterization in \cite{marivaltorta} was refined in \cite[Lem. 3]{cmmr} to the following 
\begin{proposition}[\cite{cmmr}] \label{prop-Kh}
	Let $(N,h)$ be stochastically complete. Then, for every $o\in N$ and $\lambda > 0$, there exists $\psi_0\in C^\infty(N)$ satisfying
	\begin{equation} \label{exh-func}
	\begin{cases}
	\psi_0(o) = 1, \\
	\psi_0 > 1 & \text{on } \, N\setminus\{o\}, \\
	\psi_0(x) \to +\infty & \text{as } \, x \to \infty \, \text{ on } \, N, \\
	\Delta_h \psi_0 \leq \lambda \psi_0 & \text{on } \, N.
	\end{cases}
	\end{equation}
\end{proposition}

\begin{remark}
\emph{The statement of \cite[Lem.3]{cmmr} is a bit different from that of Proposition \ref{prop-Kh}, and it is expressed in terms of an exhaustion $\varrho$. However, the statements are easily seen to be equivalent by  setting $\psi_0 = e^\varrho$. Furthermore, the characterization in \cite{marivaltorta} shows that the existence of $\psi_0$ satisfying \eqref{exh-func} is, indeed, equivalent to the stochastic completeness of $(N,h)$.
}
\end{remark}

\begin{remark}
\emph{The characterization of $(i),(ii)$ in terms of the existence of $\psi_0$ as above holds in the more general context of a potential theory for fully nonlinear operators $\mathscr{F}$. Indeed, a Liouville property for bounded subsolutions of $\mathscr{F}[u] \ge 0$ is equivalent, for large classes of $\mathscr{F}$, to the existence of suitable families of exhaustions solving $\mathscr{F}[\psi_0] \le 0$. This duality principle was studied in \cite{marivaltorta,maripessoa,maripessoa_2} and named there the AK-duality (Ahlfors-Khasminskii duality). We refer to these papers for a detailed account and for applications.
}
\end{remark}

The second auxiliary result that we need guarantees, under the assumptions of Theorem \ref{thm-CMCbound}, the possibility to embed isometrically the closure $\overline{\Sigma'}$ in a stochastically complete manifold. This will be a consequence of the next Lemma 1 in \cite{cmmr}.

\begin{proposition}[\cite{cmmr}] \label{prop-double}
Let $(N_1,h_1)$ be a Riemannian manifold and let $U_1 \subsetneq N_1$ be a connected, smooth open subset whose closure $\overline{U_1}$ in $N_1$ satisfies the following property: 
\begin{itemize}
\item[$(\wp)$] bounded subsets of $(\overline{U_1},h_1)$ have compact closure in $N_1$.
\end{itemize}
Then there exists a connected, complete Riemannian manifold $(N,h)$, a smooth open subset $U\subseteq N$ and a smooth isometry $\phi : \overline{U_1} \to \overline{U}$ such that, for every $p\in U_1$ and for every $r \geq 2\dist_{h_1}(p,\partial U_1) + 2$, 
		$$
		|B^{h}_r(\phi(p))|_{h} \leq 2|U_1 \cap B^{h_1}_{4r}(p)|_{h_1} + 5\pi.
		$$

\end{proposition}

\begin{remark}\label{rem_comple}
As observed in \cite[Rem. 7]{cmmr}, $(\wp)$ is equivalent to the completeness of the Riemannian manifold with boundary $(\overline{U_1},h_1)$.
\end{remark}

Our last lemma is technical, and we postpone its proof to the Appendix (Proposition \ref{lem_appr_appendix}).

\begin{lemma} \label{nuovo-lemma-appr}
	Let $m\geq 2$, $\kappa\geq 0$, $H\in\R$, $C\geq 0$ satisfy \eqref{appr-hp1}, \eqref{appr-hp2}, \eqref{CAHge0}.
	Then, for any $\eps>0$ there exist $A < A_1 < A+\eps$, $C < C_2 < C_1 < C+\eps$ such that
	\begin{equation} \label{appr_CAH}
		\inf\left\{ \frac{H^2}{m} - \frac{C_1 H}{t} + \left( C_2^2 - (m-1)\kappa^2 \right)\frac{t^2-1}{t^2} : t \geq A_1 \right\} > 0.
	\end{equation}
\end{lemma}

We are now ready for the

\begin{proof}[Proof of Theorem \ref{thm-CMCbound}]
	Suppose, by contradiction, that \eqref{CMC-gen-bound} is not satisfied. Then there exists $\bar x\in\Omega$ such that
	$$
		\frac{W(\bar x)}{e^{Cu(\bar x)}} > \max \left\{ A, \, \limsup_{y\to\partial\Omega} \frac{W(y)}{e^{Cu(y)}} \right\}
	$$
	and, by continuity, there exists $\eps>0$ such that
	$$
		\frac{W(\bar x)}{e^{(C+\eps)u(\bar x)}} > \max \left\{ A + \eps, \, \limsup_{y\to\partial\Omega} \frac{W(y)}{e^{Cu(y)}} \right\}.
	$$
	We fix $A < A_1 < A+\eps$, $C < C_2 < C_1 < C+\eps$, $\eps_0 > 0$ such that
	\begin{equation} \label{CAHeps0}
		\frac{H^2}{m} - \frac{C_1 H}{t} + \left( C_2^2 - (m-1)\kappa^2 \right) \frac{t^2-1}{t^2} \geq \eps_0 \qquad \text{for every } \, t \geq A_1.
	\end{equation}
	This is possible by Lemma \ref{nuovo-lemma-appr}. Set
	$$
		z = \frac{W}{e^{C_1 u}}.
	$$
	Since $C_1 < C + \eps$ and $u \geq 0$, we have $z(\bar x) \geq W(\bar x) e^{-(C+\eps)u(\bar x)}$. Then, as $A_1 < A + \eps$, there exists a regular value $\gamma\in\R$ for $z$ satisfying
	$$
		\max \left\{ A_1, \limsup_{x\to\partial\Omega} \frac{W(x)}{e^{Cu(x)}} \right\} < \gamma < \sup_\Omega z.
	$$
Fix a connected component $\Omega_\gamma$ of the non-empty set $\{ x \in \Omega : z(x) > \gamma\}$, let $\Sigma_\gamma$ be the graph of $u$ over $\Omega_\gamma$, and fix $p \in \Sigma_\gamma$. Note that $z \leq We^{-Cu}$ since $C_1 u \geq Cu$, so
	$$
		\limsup_{y\to\partial\Omega} z(y) \leq \limsup_{y\to\partial\Omega} \frac{W(y)}{e^{Cu(y)}} < \gamma
	$$
and then the closure $\overline{\Omega_\gamma}$ in $M$ is a smooth manifold with boundary contained in $\Omega$. Observe also that the closure $\overline{\Sigma_\gamma}$ in $\Sigma$ is the smooth manifold with boundary $\pi^{-1}(\overline{\Omega_\gamma})$. We claim that $(\wp)$ in Proposition \ref{prop-double} is satisfied. Let $U\subseteq \overline{\Sigma_\gamma}$ be bounded. Since the projection $\pi : (\Sigma,g) \to (\Omega,\sigma)$ does not increase distances, $\pi(U)$ is bounded in $(\overline{\Omega_\gamma}, \sigma)$, hence in $M$. The completeness of $M$ implies that its closure $\overline{\pi(U)} \subset \overline{\Omega_\gamma}$ in $M$ is compact. Pulling back using $\pi$ we deduce that $U$ has compact closure in $\overline{\Sigma_\gamma}$ (hence, in $\Sigma$). By \eqref{vol-gr-Sigma},
	$$
		\liminf_{r\to \infty} \frac{\log|\Sigma_{\gamma} \cap B^g_r(p)|_g}{r^2} \leq \liminf_{r\to \infty} \frac{\log|B^g_r(p)|_g}{r^2} < \infty.
	$$
Observing that $\Sigma_\gamma$ is connected by construction, by Proposition \ref{prop-double} there exists a complete, connected Riemannian manifold $(N,h)$ satisfying
	\begin{equation} \label{vol-N}
		\liminf_{r\to \infty} \frac{\log |B^h_r(o)|_h}{r^2} < \infty
	\end{equation}
	for some point $o\in N$, and an isometry $\phi : \Sigma_{\gamma} \to U$ between $\Sigma_{\gamma}$ and an open subset $U\subseteq N$ such that $\phi(p) = o \in U$. By \eqref{vol-N}, the complete manifold $(N,h)$ is stochastically complete. Pick
	\begin{equation} \label{lambda}
		\lambda \in \left( 0, C_1^2\frac{A_1^2-1}{A_1^2} \right).
	\end{equation}
	and let $\psi_0 \in C^\infty(N)$ satisfy \eqref{exh-func} in Proposition \ref{prop-Kh}. Then, the function $\psi_1 = \psi_0\circ\phi \in C^\infty(\Sigma_{\gamma})$ satisfies
	$$
		\begin{cases}
			\psi_1(p) = 1, \\
			\psi_1 > 1 & \text{on } \, \Sigma_{\gamma} \setminus\{p\}, \\
			\disp \psi_1(x) \to +\infty & \text{as } \, r_g(x) \to \infty, \\
			\Delta_g \psi_1 \leq \lambda \psi_1 & \text{on } \, \Sigma_{\gamma},
		\end{cases}
	$$
where $r_g$ is the distance from $p$ in $\Sigma_\gamma$. Let $\beta>0$ be such that
	\begin{equation} \label{CAHbl}
		\frac{C_1^2}{1+2\beta} \geq C_2^2, \qquad \frac{C_1^2}{1+\beta}\frac{A_1^2-1}{A_1^2} - \beta\lambda > \frac{C_1^2}{1+2\beta}\frac{A_1^2-1}{A_1^2}.
	\end{equation}
	The first condition holds for any sufficiently small $\beta>0$ as $C_1 > C_2$, while the second one is equivalent to
	$$
		\lambda < \frac{C_1^2}{(1+\beta)(1+2\beta)}\frac{A_1^2-1}{A_1^2},
	$$
	and by \eqref{lambda} this is true for any sufficiently small $\beta>0$. Hence, it is possible to find $\beta>0$ satisfying \eqref{CAHbl}. Since $\R^+ \ni t \mapsto (t^2-1)/t^2$ is nondecreasing, \eqref{CAHbl} implies
	\begin{equation} \label{CAHbl+}
		\frac{C_1^2}{1+\beta}\frac{t^2-1}{t^2} - \beta\lambda > \frac{C_1^2}{1+2\beta}\frac{t^2-1}{t^2} \geq C_2^2 \frac{t^2-1}{t^2} > 0 \qquad \text{for every } \, t \geq A_1.
	\end{equation}
	The second inequality, together with \eqref{CAHeps0}, gives
	\begin{equation} \label{CAHbl++}
		\frac{H^2}{m} - \frac{C_1 H}{t} + \left( \frac{C_1^2}{1+2\beta} - (m-1)\kappa^2 \right) \frac{t^2 - 1}{t^2} \geq \eps_0 \qquad \text{for every } \, t \geq A_1.
	\end{equation}

	Let $\delta>0$ be such that
	\begin{equation} \label{def-delta}
		\gamma < W(p) (e^{-C_1 u(p)} - \delta), \qquad C_1 H \delta < \eps_0.
	\end{equation}
The existence of such $\delta$ is guaranteed since $W(p)e^{-C_1 u(p)} = z(p) > \gamma$. We define
	$$
		\psi = \psi_1^{-\beta}, \qquad \eta = \psi e^{-C_1 u} - \delta, \qquad \tilde{z} = W\eta.
	$$
	Since $C_1u \geq 0$, when $r_g(x) \ra \infty$ we infer $\psi(x) \to 0$ and thus $\eta(x) \to -\delta$, so the set $\{ x \in \Sigma_{\gamma} : \eta(x) > 0 \}$ is bounded in $(\Sigma,g)$. Hence, 
	$$
		\tilde{\Sigma}_\gamma : = \{ x \in \Sigma_{\gamma} : \tilde{z}(x) > \gamma \}
	$$
	is a bounded subset of $\Sigma_\gamma$ and therefore has compact closure in $\Sigma$. Moreover, $p \in \tilde\Sigma_\gamma$ by \eqref{def-delta} and $\psi_1(p)=1$. We also observe that
	$$
		\tilde{z} \leq W(e^{-C_1 u} - \delta) = z - \delta W \leq z - \delta,
	$$
	so
	$$
		\overline{\tilde\Sigma_\gamma} \subseteq \overline{\{ x \in \Sigma : z(x) > \gamma + \delta \}} \subseteq \{ x \in \Sigma : z(x) \geq \gamma + \delta \} \subseteq \Sigma_{\gamma},
	$$
	that is, the boundary of $\tilde\Sigma_\gamma$ is contained in $\Sigma_{\gamma}$. As $\Sigma_{\gamma}$ is connected and $\tilde\Sigma_\gamma$ is open, non-empty, and with closure contained in $\Sigma_{\gamma}$, it follows that $\partial\tilde\Sigma_\gamma \neq \emptyset$ and, by continuity of $\tilde{z}$, we have
	$$
		\tilde{z} = \gamma \qquad \text{on } \, \partial\tilde\Sigma_\gamma.
	$$
	By these observations, $\tilde{z}$ attains its maximum over the compact set $\overline{\tilde\Sigma_\gamma}$ at an interior point $x_0 \in \tilde\Sigma_\gamma$. By \eqref{LLWz+1}, we have
	\[
	\begin{array}{lcl}	
	\LL_W \tilde{z} & \geq & \disp \left( \frac{H^2}{m} - \frac{C_1 H}{W} - (m-1)\kappa^2 \frac{W^2-1}{W^2} \right. \\[0.5cm]
	& & \disp \left. + \left(1+\frac{\delta}{\eta}\right) \left( \frac{C_1^2}{1+\beta} \frac{W^2-1}{W^2} - \beta\lambda \right) - \frac{C_1 H \delta}{\tilde{z}} \right) \tilde{z}
	\end{array}
	\]
	on $\tilde\Sigma_\gamma$. We claim that
	\begin{equation} \label{claim-CAH}
		\frac{H^2}{m} - \frac{C_1 H}{W} - (m-1)\kappa^2 \frac{W^2-1}{W^2} + \left(1+\frac{\delta}{\eta}\right) \left( \frac{C_1^2}{1+\beta} \frac{W^2-1}{W^2} - \beta\lambda \right) - \frac{C_1 H \delta}{\tilde{z}} > 0
	\end{equation}
	in $\tilde\Sigma_\gamma$. This would imply that $\LL_W \tilde{z} > 0$ at the interior maximum point $x_0$, thus yielding the desired contradiction and concluding the proof.
	
	We prove the claim. First, observe that
	$$
		W > \tilde{z} > A_1 \qquad \text{in } \, \tilde\Sigma_\gamma.
	$$
	By \eqref{CAHbl+}, we have
	$$
		\left(1+\frac{\delta}{\eta}\right) \left( \frac{C_1^2}{1+\beta} \frac{W^2-1}{W^2} - \beta\lambda \right) \geq \frac{C_1^2}{1+\beta} \frac{W^2-1}{W^2} - \beta\lambda > \frac{C_1^2}{1+2\beta} \frac{W^2-1}{W^2}.
	$$
	Therefore, the LHS of \eqref{claim-CAH} is larger than
	$$
		\frac{H^2}{m} - \frac{C_1 H}{W} + \left( \frac{C_1^2}{1+2\beta} - (m-1)\kappa^2 \right)\frac{W^2-1}{W^2} - \frac{C_1 H\delta}{\tilde{z}}
	$$
	and by \eqref{CAHbl++},
	$$
		\frac{H^2}{m} - \frac{C_1 H}{W} + \left( \frac{C_1^2}{1+2\beta} - (m-1)\kappa^2 \right)\frac{W^2-1}{W^2} - \frac{C_1 H\delta}{\tilde{z}} \geq \eps_0 - \frac{C_1 H \delta}{\tilde{z}}.
	$$
	If $C_1 H \delta \leq 0$ then we conclude
	$$
		\eps_0 - \frac{C_1 H \delta}{\tilde{z}} \geq \eps_0 > 0 \qquad \text{in } \, \tilde\Sigma_\gamma.
	$$
	If $C_1 H \delta > 0$, then we use the fact that $\tilde{z} > A_1 > 1$ in $\tilde\Sigma_\gamma$ together with \eqref{def-delta} to get
	$$
		\eps_0 - \frac{C_1 H \delta}{\tilde{z}} \geq \eps_0 - C_1 H \delta > 0 \qquad \text{in } \, \tilde\Sigma_\gamma.
	$$
	In both cases, we obtain the claimed validity of \eqref{claim-CAH} on $\tilde\Sigma_\gamma$.
\end{proof}

\section{Splitting of capillary graphs}

\subsection{Monotonicity of solutions in presence of Killing vectors}

As a first step towards the proof of Theorem \ref{teo_splitting}, we prove the following

\begin{proposition} \label{prop-mono}
	Let $(M,\sigma)$ be a complete Riemannian manifold and $\Omega \subseteq M$ a connected open set with smooth boundary and parabolic closure. Let $u \in C^2(\overline{\Omega})$ satisfy
	$$
		\sup_\Omega |Du| < \infty, \qquad \qquad \diver \left( \frac{Du}{\sqrt{1+|Du|^2}} \right) = H \qquad \text{on } \, \Omega
	$$
	for some constant $H\in\R$. If $X$ is a Killing vector field on $\overline{\Omega}$ satisfying
	$$
		\left\{\begin{array}{r@{\;}ll}
			\sup_\Omega |X| & < \infty, & \\[0.2cm]
			(Du,X) & \geq 0 & \text{on } \, \partial \Omega,
		\end{array}\right.
	$$
	then
	$$
		(Du,X) \geq 0 \qquad \text{on } \, \Omega.
	$$
	Moreover,
	\begin{align*}
		(Du,X) > 0 \quad \text{on } \, \Omega & \qquad \text{if} \qquad (Du,X) \not\equiv 0 \quad \text{on } \, \partial\Omega, \\
		(Du,X) \equiv 0 \quad \text{on } \, \Omega & \qquad \text{if} \qquad (Du,X) \equiv 0 \quad \text{on } \, \partial\Omega.
	\end{align*}
\end{proposition}

To prove the proposition, we shall recall some facts about the parabolicity of weighted operators on manifolds with boundary, adapted from \cite{imperapigolasetti}, which deals with the case of the Laplace-Beltrami operator. First, given a manifold with boundary $(N,h)$ and $f \in C^1(N)$, we define the weighted Laplacian 
	\[
	\Delta_f \phi : =  e^{f} \diver\left( e^{-f} \nabla \phi \right) \qquad \forall \phi \in C^2(N),
	\]  
and observe that $\Delta_f$ is symmetric if we integrate function with compact support in $\mathrm{int} \, N$  with respect to the weighted measure $e^{-f}\di x_h$. We say that $\Delta_f$ is parabolic on $(N,h)$ if, for every compact set $K \subset N$ with non-empty interior, the capacity
	\[
	\capac_f(K) : = \inf \left\{ \int_N |\nabla \phi|^2 e^{-f} \di x_h \ : \ \phi \in \lip_c(N), \ \phi \ge 1 \ \text{ on } \, K \right\} = 0.
	\]
By definition, it readily follows that if $N, N'$ are smooth manifolds of the same dimension with boundary, and $N' \subset N$ is closed, then for each $f \in C^\infty(N)$ it holds
	\[
	\Delta_f \ \ \text{ is parabolic on } \, N \qquad \Longrightarrow \qquad \Delta_f \ \ \text{ is parabolic on } \, N' 
	\]

The following characterization is showed in \cite[Thm. 1.5 and Thm. 0.10]{imperapigolasetti} for $\Delta_f = \Delta$, but its proof extends verbatim to weighted operators. The result is stated for smooth boundary, but $C^2$-regularity suffices.

\begin{theorem}[\cite{imperapigolasetti}]\label{teo_IPS}
Let $(N,h)$ be a smooth manifold with $C^2$-boundary, and let $f \in C^1(N)$. Then, the following properties are equivalent:
\begin{itemize}
	\item[(i)] $\Delta_f$ is parabolic on $(N,h)$.
	\item[(ii)] every $v \in C(N) \cap W^{1,2}_\loc(N)$ solving 
	\begin{equation}\label{eq_v}
		\left\{\begin{array}{r@{\;}ll}
			\Delta_f v & \geq 0 & \quad \text{on } \, \mathrm{int} \, N, \\[0.1cm]
			\partial_{\eta} v & \leq 0 & \quad \text{on } \, \partial N, \\
			\sup_N v & < \infty
		\end{array}\right.
	\end{equation}
is constant on each connected component of $N$, where $\eta$ is the exterior normal of $\partial N \hookrightarrow N$.	
\end{itemize}
In particular, if any of $(i)$ or $(ii)$ holds, every solution $\psi \in W^{1,2}_\loc(N) \cap C(N)$ of 
	\[
	\Delta_f \psi \ge 0 \quad \text{on } \, N, \qquad \sup_{N} \psi < \infty
	\]
satisfies
	\begin{equation}\label{supsup}
	\sup_N \psi = \sup_{\partial N} \psi.
	\end{equation}
\end{theorem}

\begin{remark}
\emph{We recall that $v$ is a weak solution of $\Delta_f v \ge 0$ on $N$, $\partial_\eta v \le 0$ on $\partial N$ if and only if 
$$
	\int_N \langle \nabla v, \nabla \varphi \rangle e^{-f} \di x \leq 0 \qquad \text{for every } \, 0 \leq \varphi \in C^\infty_c(N).
$$
}
\end{remark}

The next Lemma relates the parabolicity of ($\Delta$ on) a boundaryless manifold $N$ with that of the product $N \times I$, with $I \subset \R$ a closed interval. We recall from the Introduction that a smooth open set $\Omega\subseteq (M,\sigma)$ has a parabolic closure  if the Laplacian $\Delta$ is parabolic on the manifold with boundary $(\overline{\Omega} = \Omega\cup\partial\Omega,\sigma)$, according to the above definition.\par

\begin{lemma}\label{lem_paraprod}
Let $N$ be a manifold without boundary, and let $I \subset \R$, $I \not \equiv \R$ be a closed interval. Then, 
	\[
	\text{$I \times N$ is parabolic} \qquad \Longrightarrow \qquad \text{$N$ is parabolic}.
	\]
\end{lemma}

\begin{proof}
Up to translation and reflection, we can assume that either $I = [0, \infty)$ or $I = [0,T]$, for some $T \in \R^+$. Furthermore, as we observed before, if $[0, \infty) \times N$ is parabolic, then every smooth open subset of it has a parabolic closure, in particular $[0,T] \times N$ is parabolic. Therefore, it suffices to consider $I = [0,T]$. Given a compact set $C \subset N$, consider $K = [0,T]\times C$. Since $I\times N$ is parabolic, there exists a sequence $\{\varphi_j\} \subset \lip_c(I \times N)$ such that $\varphi \ge 1$ on $K$ and 
	\[
	0 = \lim_{j \to \infty} \int_{I \times N} |D \varphi_j|^2 \di x \di t = \lim_{j \to \infty} \int_N \left[ \int_0^T \left(\frac{\partial \varphi_j}{\partial t}\right)^2 + |D_N \varphi_j|^2 \di t \right] \di x
	\]
with $D_N$ and $\di x$ the gradient and volume measure of $N$. Setting 
	\[
	\bar \varphi_j(x) = \int_0^T \varphi_j(x,t) \di t \ \ \in \lip_c(N),
	\]
note that $\bar \varphi_j \ge 1$ on $C$ and, by Cauchy-Schwarz inequality, 
	\[
	|D_N \bar \varphi_j|^2 = \left| \int_0^T D_N \varphi_j \di t \right|^2 \le T \int_0^T |D_N \varphi_j|^2 \di t.
	\]
Therefore,	
	\[
	0 = \lim_{j \to \infty} \int_N \left[ \int_0^T \left(\frac{\partial \varphi_j}{\partial t}\right)^2 + |D_N \varphi_j|^2 \di t \right] \di x \ge \frac{1}{T} \lim_{j \to \infty} \int_N |D \bar \varphi_j|^2 \di x,
	\]
so $N$ is parabolic. 
\end{proof}

The next lemma shows that if an open subset $\Omega \subset (M, \sigma)$ has a parabolic closure, then for any given $u\in C^\infty(\Omega)$ the differential operator $\LL_W$ defined in \eqref{LLW-def} is parabolic on the graph $(\Sigma,g)$ of $u$. Remarkably, this implication holds also without requiring $\sup_\Omega |Du| < \infty$, and contrasts with the case of the graph Laplacian $\Delta_g$, which may not inherit parabolicity from the base domain $\Omega$ if $u$ has unbounded gradient.

\begin{lemma} \label{lem-LLW-weak-par}
	Let $(M,\sigma)$ be a complete Riemannian manifold and let $\Omega \subseteq M$ be a connected open subset with $C^2$-boundary and parabolic closure. For $u\in C^\infty(\Omega)$, let $(\Sigma, g)$ be the graph of $u$. Then, for each open subset $\Omega'$ with $C^2$-boundary and satisfying $\overline{\Omega'} \subset \Omega$, $\LL_W$ is parabolic on the graph $\overline{\Sigma'}$ over $\overline{\Omega'}$. In particular, for every $v \in C^\infty(\Omega)$ satisfying
		\[
		\LL_W v \ge 0 \quad \text{on } \, \Omega, \qquad \sup_\Omega v < \infty, 
		\]
Then
		\[
		\sup_\Omega v = \limsup_{x \to \partial \Omega} v(x).
		\]
\end{lemma}



\begin{proof}[Proof of Lemma \ref{lem-LLW-weak-par}]
For any compact set $K\subseteq\overline{\Omega'}$, we define the capacities
	\begin{align*}
		\mathrm{cap}_\sigma(K) & = \inf\left\{ \int_\Omega |D\varphi|^2 \di x : \varphi \in \lip_c(\overline{\Omega'}), \, \varphi \geq 1 \, \text{ on } \, K \right\}, \\
		\mathrm{cap}_{g,W}(K) & = \inf\left\{ \int_\Omega \|\nabla\phi\|^2 \di x_W : \phi \in \lip_c(\overline{\Omega'}), \, \varphi \geq 1 \, \text{ on } \, K \right\}.
	\end{align*}
Since $\di x_W = W^{-1} \di x$, $W \geq 1$ and $\|\nabla\varphi\| \leq |D\varphi|$ for every $\varphi\in C^1(\Omega)$, we deduce $\mathrm{cap}_{g,W}(K) \leq \mathrm{cap}_\sigma(K)$. The inclusion $\Omega' \subset \Omega$ implies that $\Omega'$ is parabolic, thus $\mathrm{cap}_{g,W}(K)=0$. Hence, $(\overline{\Sigma'},g)$ is parabolic. To conclude, suppose, by contradiction, that there exists $v \in C^\infty(\Omega)$ satisfying
	$$
		\LL_W v \geq 0 \quad \text{in } \, \Omega, \qquad \limsup_{x\to\partial\Omega} v(x) < \sup_\Omega v < \infty.
	$$
Fix a regular value $\gamma\in\R$ of $v$ such that
	$$
		\limsup_{x\to\partial\Omega} v(x) < \gamma < \sup_\Omega v,
	$$
and observe that $\Sigma_\gamma = \{ x \in \Sigma : v(x) > \gamma \}$ has smooth, non-empty boundary, and that $\overline{\Sigma_\gamma}\subset \Sigma$. Hence $\Sigma_\gamma$ has parabolic closure. However, $v$ is a 
solution of
	\begin{equation}\label{eq_v}
		\left\{\begin{array}{r@{\;}ll}
			\Delta_f v & \geq 0 & \quad \text{in } \, \Sigma_\gamma, \\[0.1cm]
			\partial_{\eta} v & < 0 & \quad \text{on } \, \partial \Sigma_\gamma, \\
			\sup_N v & < \infty,
		\end{array}\right.
	\end{equation}
that is not constant on some connected component of $\overline{\Sigma_\gamma}$, contradicting Theorem \ref{teo_IPS}.
\end{proof}

Proposition \ref{prop-mono} now follows at once.

\begin{proof}[Proof of Proposition \ref{prop-mono}]
	By elliptic regularity, $u$ is smooth on $\Omega$. Therefore, in our assumptions, $\bar v := (Du, X) \in C^\infty(\Omega) \cap C^2(\overline{\Omega})$ satisfies
	$$
		\sup_\Omega |\bar v| \leq \left( \sup_\Omega |Du| \right) \left( \sup_\Omega |X| \right) < \infty
	$$
	and by \eqref{LLWvbar} we have
	$$
		\LL_W \bar v = 0 \qquad \text{in } \, \Omega.
	$$
	By applying Lemma \ref{lem-LLW-weak-par} to both functions $\bar v$ and $-\bar v$ we deduce
	$$
		\inf_\Omega \bar v = \inf_{\partial\Omega} \bar v, \qquad \sup_\Omega \bar v = \sup_{\partial\Omega} \bar v.
	$$
	In particular, $\bar v \geq 0$ on $\Omega$, and we further have $\bar v \not\equiv 0$ on $\Omega$ if and only if $\bar v \not\equiv 0$ on $\partial\Omega$. By the strong maximum principle for the elliptic operator $\LL_W$, if $\bar v \ge 0$ and does not vanish identically, then $\bar v > 0$ on $\Omega$. This concludes the proof.
\end{proof}

\subsection{Splitting of monotone solutions}

The goal of this section is to prove the following

\begin{proposition}\label{prop_splitting}
	Let $(M^m, ( \, , \, ))$ be a complete Riemannian manifold, and let $\Omega\subseteq M$ be a connected open subset with smooth boundary and unit exterior normal $\bar \eta$. Assume that $\overline\Omega$ is parabolic and
	\[
	\Ricc \ge 0 \qquad \text{on } \, \Omega. 
	\]
Split $\partial \Omega$ into its connected components $\{\partial_j\Omega\}$, $1 \le j \le j_0$, possibly with $j_0 = \infty$. Let $u \in C^2(\overline\Omega)$ be a solution of the capillarity problem
	\begin{equation} \label{capill-eq-sec}
		\begin{cases}
			\diver \left( \dfrac{Du}{\sqrt{1+|Du|^2}} \right) = H & \qquad \text{on } \, \Omega, \\
			u = b_j, \ \partial_{\bar\eta} u = c_j & \qquad \text{on } \, \partial_j \Omega, \; 1 \leq j \leq j_0, \\
			\inf_\Omega u > -\infty,					
			\end{cases}
	\end{equation}
	for some set of constants $H, c_j, b_j \in \R$, with the agreement $b_1 \le b_2 \le \ldots \le b_{j_0}$. Suppose that 
	\[
	\begin{array}{l}
	\sup_\Omega |Du| < \infty, \\[0.2cm]
	(Du,X) > 0 \quad \text{on $\Omega$, for some Killing field $X$ on $\overline\Omega$}
	\end{array}
	\]	
Then, 
	\begin{itemize}
		\item[(i)] $\Omega = (0, T) \times N$ with the product metric, for some $T \le \infty$ and some complete, boundaryless, parabolic manifold $N$ with $\Ricc_N \ge 0$,
		\item[(ii)] the product $(X, \partial_t)$ is a positive constant on $\Omega$.
		\item[(iii)] $c_1 \le 0$ ($c_1 < 0$ if $H \le 0$) and, denoting with $\partial_1 \Omega = \{0\} \times N$, $u(t,x) = b_1 + u_{c_1,H}(t)$ is a (translated) standard example, with $u_{c_1,H}(t)$ as in \eqref{eq_split}.
	\end{itemize}
\end{proposition}

Combining the above Propositions \ref{prop-mono} and \ref{prop_splitting}, we readily deduce our main Theorem \ref{teo_splitting} 

\begin{proof}[Proof of Theorem \ref{teo_splitting}]
Because of the gradient estimate in Theorem \ref{intro-prop-bound}, each one of the assumptions in \eqref{intro-dOm-hp} together with the boundedness of $\{c_j\}$ guarantees that 
	\[
	\sup_\Omega |Du| < \infty.
	\]
Furthermore, \eqref{ipo_X2} rewrites as $(Du,X) \ge 0$ and $\not \equiv 0$ along $\partial \Omega$. Therefore, we can apply Proposition \ref{prop-mono} to deduce that $(Du,X) >0$ on $\Omega$. Proposition \ref{prop_splitting} then gives the desired conclusions. 	
\end{proof}

As remarked in the Introduction, the proof of Proposition \ref{prop_splitting} relies on a geometric weighted Poincar\'e inequality. In its proof, we need the next fundamental identity, inspired by \cite{far_vald_ARMA} and by more general computations in \cite{far_fra_mar}.

\begin{lemma}\label{lem_ideboundary}
	Let $\Omega \subset M$ have smooth boundary with exterior normal $\eta$, and let $\Sigma$ be the graph of $u : \Omega \ra \R$, with $u \in C^2(\overline\Omega)$. Assume that $u$ and $|Du|$ are locally constant on $\partial \Omega$. Let $X$ be a Killing vector field on $\overline\Omega$ and set $\bar v = (Du,X)$. Then
	\begin{equation}\label{ide_boundary}
		\langle W \|\nabla u\|^2 \nabla \bar v, \eta \rangle = \langle \bar v \nabla W, \eta \rangle \qquad \text{on } \, \partial \Sigma,
	\end{equation}
\end{lemma}

\begin{proof}
	Let $\partial_i\Omega$ be a connected component of $\partial\Omega$ and let $\partial_i\Sigma \subseteq \partial\Sigma$ be the the graph of $u$ over $\partial_i\Omega$. If $\di u = 0$ along $\partial_i\Omega$ then $\nabla u = 0$ along $\partial_i\Sigma$, thus $\|\nabla u\| = 0 = (Du,X) = \bar v$ and both sides of \eqref{ide_boundary} vanish. If $\di u \neq 0$, then $\eta$ is parallel to $\nabla u$ and $\bar\eta = Du/|Du|$ is a unit vector field perpendicular to $\partial_i\Omega$ in $M$. Recalling that
	$$
		\|\nabla u\|^2 = \frac{W^2-1}{W^2}
	$$
	and rearranging terms we see that validity of \eqref{ide_boundary} on $\partial_i\Sigma$ is equivalent to saying that the function
	$$
		\zeta := \frac{\bar v}{\sqrt{W^2-1}} \equiv \frac{(Du,X)}{|Du|} \equiv (\bar\eta,X)
	$$
	satisfies $(D\zeta,\bar\eta) = 0$ on $\partial_i\Omega$. Indeed, we have
	$$
	    \nabla \zeta = \frac{W}{(W^2-1)^{3/2}} \left( \frac{W^2-1}{W} \nabla \bar v - \bar v \nabla W \right) = \frac{W}{(W^2-1)^{3/2}} \left( W \|\nabla u\|^2 \nabla \bar v - \bar v \nabla W \right),
	$$
	so \eqref{ide_boundary} rewrites as $\langle \nabla \zeta,\nabla u\rangle = 0$. By \eqref{D_phi_nabla_phi}, this is equivalent to $(D\zeta,D u) = 0$, hence to $(D\zeta,\bar\eta) = 0$. To show that $(D\zeta,\bar\eta) = 0$ on $\partial_i\Omega$, we differentiate $\zeta$ to get
	$$
		(D\zeta,\bar\eta) = \bar \eta(\zeta) = (D_{\bar\eta} \bar\eta,X) + (\bar\eta,D_{\bar\eta} X).
	$$
	By the Killing condition we have $(\bar\eta,D_{\bar\eta} X) = 0$. From the differential identity
	$$
		\di|Du|^2 = 2D^2u(Du,\,\cdot\,)
	$$
	we infer $\di|Du| = D^2u(\bar\eta,\,\cdot\,)$, hence
	\begin{equation} \label{D_eta_X_boundary}
	    \begin{split}
		    (D_{\bar\eta}\bar\eta,X) & = \frac{1}{|Du|}(D_{\bar\eta} Du,X) + \left(D_{\bar\eta} \frac{1}{|Du|}\right)(Du,X) \\
		    & = \frac{1}{|Du|}D^2 u(\bar\eta,X) - \frac{1}{|Du|^2} D^2u(\bar\eta,\bar\eta)(Du,X) \\
		    & = \frac{1}{|Du|}\left( D^2 u(\bar\eta,X) - D^2u(\bar\eta,\bar\eta)(\bar\eta,X) \right).
	    \end{split}
	\end{equation}
	Since $|Du|$ is constant on $\partial_i\Omega$, we have $(D|Du|,Y) = 0$ on $\partial_i\Omega$ for any vector $Y$ satisfying $(Y,\bar\eta) = 0$. Hence, $(D|Du|,X) = (D|Du|,\bar\eta)(\bar\eta,X)$ and we have
	$$
		D^2 u(\bar\eta,X) = (D|Du|,X) = (D|Du|,\bar\eta)(\bar\eta,X) = D^2 u(\bar\eta,\bar\eta) (\bar\eta,X)
	$$
	so from \eqref{D_eta_X_boundary} we get $(D_{\bar\eta} \bar\eta,X) = 0$. Therefore, we conclude that $(D\zeta,\bar\eta) \equiv 0$ on $\partial_i\Omega$.
\end{proof}

\begin{remark}
	\emph{In the proof we see that, when $\di u\neq 0$, \eqref{ide_boundary} follows from $(\bar\eta,D_{\bar\eta} X) = 0$ and $D_{\bar\eta} \bar\eta = 0$. The first condition is verified since $X$ is a Killing vector, while the second one amounts to saying that the integral curves of $Du$ have zero geodesic curvature at points of $\partial\Omega$, a consequence of $|Du|$ and $u$ being locally constant on $\partial\Omega$. The conclusion then parallels the fact that the angle between a Killing vector field and the tangent vector of a given geodesic curve remains constant along the curve.}
\end{remark}

\begin{lemma}[Geometric Poincar\'e formula]\label{lem_poinc}
	Let $u \in C^2(\overline\Omega)$ satisfy 
		\[ 
		\left\{ \begin{array}{l}
		\diver \left( \frac{Du}{\sqrt{1+|Du|^2}} \right) = H \qquad \text{on } \, \Omega, \\[0.5cm]
		u, \partial_{\bar \eta} u \qquad \text{locally constant on } \, \partial \Omega, 
		\end{array}\right.
		\]
where $H \in \R$ and $\bar \eta$ is the outward pointing unit normal of $\partial \Omega \hookrightarrow \Omega$. Assume that $u$ is monotone in the direction of a Killing field $X$ on $\overline{\Omega}$, with $\bar v : = (Du,X) > 0$ on $\Omega$. Then, for every $\varphi \in \lip_c(\overline{\Omega})$,
	\begin{equation}\label{eq_poinc}
	\begin{array}{l}
		\disp \int_\Sigma \left[ W^2 \left( \|\nabla_\top \|\nabla u\|\|^2 + \|\nabla u\|^2 \|A\|^2 \right) + \frac{\Ricc(Du,Du)}{W^2} \right] \varphi^2 \di x_g + \\[0.5cm]
		\qquad \qquad \disp + \int_\Sigma \frac{\bar v^2}{W^2} \left\| \nabla \left( \frac{\varphi\|\nabla u\|W}{\bar v}\right) \right\|^2\di x_g \le \int_\Sigma \|\nabla u\|^2 \|\nabla \varphi \|^2 \di x_g
	\end{array}
	\end{equation}
\end{lemma}

\begin{proof}
We recall that $\bar v$ satisfies $\LL_W \bar v = 0$. Thus, for every $\eps>0$,
\begin{equation}\label{eq_bonita}
	\LL_W \log(\bar v + \eps) = - \| \nabla \log(\bar v + \eps)\|^2.
\end{equation}
We hereafter consider integration with respect to the weighted measures $\di x_W$ and $\di \mathscr{H}^{m-1}_W$, defined by 
$$
	\di x_W = W^{-2} \di x_g, \qquad \di \mathscr{H}^{m-1}_W = W^{-2} \di \mathscr{H}^{m-1}_g,
$$
that we omit to write. Recall that $\LL_W$ is symmetric with respect to $\di x_W$. Integrating by parts on $\Sigma$ against $\phi^2$, with $\phi \in \lip_c(\overline{\Omega})$, we obtain 
	\[
		\int_{\partial \Sigma} \phi^2 \langle \frac{\nabla \bar v}{\bar v +\eps}, \eta \rangle = 2\int_\Sigma \frac{\phi}{\bar v +\eps} \langle \nabla \phi, \nabla \bar v \rangle - \int_\Sigma \phi^2\|\nabla \log (\bar v +\eps)\|^2,
	\]
	where $\eta$ is the exterior normal to $\partial\Sigma$ in $\Sigma$. Direct computation gives
	$$
		(\bar v + \eps)^2 \left\| \nabla \left( \frac{\phi}{\bar v + \eps}\right) \right\|^2 = \|\nabla\phi\|^2 - 2\frac{\phi}{\bar v +\eps} \langle \nabla \phi, \nabla \bar v \rangle + \phi^2\|\nabla \log (\bar v +\eps)\|^2,
	$$
	so we deduce
	\begin{equation} \label{picone}
		\int_{\partial \Sigma} \phi^2 \langle \frac{\nabla \bar v}{\bar v+\eps}, \eta \rangle = \int_\Sigma \|\nabla \phi\|^2 - \int_\Sigma (\bar v + \eps)^2 \left\| \nabla \left( \frac{\phi}{\bar v + \eps}\right) \right\|^2 \qquad \forall \phi \in \lip_c(\overline{\Omega}),
	\end{equation}
	that is a form of Picone's identity, \cite{picone}. Let $\varphi \in \lip_c(\overline{\Omega})$ be given. Recall that $\LL_W W = q W$, where we have set for convenience
	$$
		q := \|\II_\Sigma\|^2 + \overline{\Ricc}({\bf n},{\bf n}).
	$$
	We multiply both sides of $\LL_W W = q W$ by $\varphi^2 W \bar v/(\bar v+\eps)$ and we integrate by parts to get
	\begin{equation}\label{eq_variecaccio}
	\begin{array}{lcl}
	\disp \int_\Sigma q\varphi^2 \frac{\bar v}{\bar v+\eps}W^2 & = & \disp \int_\Sigma \varphi^2 \frac{\bar v}{\bar v+\eps} W \LL_W W \\[0.5cm]
	& = & \disp \int_{\partial \Sigma} \varphi^2 \frac{\bar v}{\bar v+\eps} W \langle \nabla W, \eta \rangle - 2 \int_\Sigma \varphi \frac{\bar v}{\bar v+\eps} W \langle \nabla \varphi, \nabla W \rangle \\[0.5cm]
	& & \disp - \int_\Sigma \varphi^2 \frac{\bar v}{\bar v+\eps} \|\nabla W\|^2 - \int_\Sigma \varphi^2 W \langle \frac{\eps \nabla \bar v}{(\bar v+\eps)^2}, \nabla W \rangle 	\\[0.5cm]
	& = & \disp \int_{\partial \Sigma} \varphi^2 \frac{\bar v}{\bar v+\eps} W \langle \nabla W, \eta \rangle - 2 \int_\Sigma \varphi W \langle \nabla \varphi, \nabla W \rangle \\[0.5cm]
	& & \disp - \int_\Sigma \varphi^2 \|\nabla W\|^2 - \int_\Sigma \varphi^2 W \langle \frac{\eps \nabla \bar v}{(\bar v+\eps)^2}, \nabla W \rangle \\[0.5cm]
	& & \disp + 2 \int_\Sigma \varphi W \frac{\eps}{\bar v+\eps} \langle \nabla \varphi, \nabla W \rangle +  \int_\Sigma \varphi^2 \frac{\eps}{\bar v + \eps} \|\nabla W\|^2.
	\end{array}
	\end{equation}
By Lemma \ref{lem_ideboundary}, $\langle \bar v W \nabla W, \eta \rangle = W^2 \|\nabla u\|^2 \langle \nabla \bar v, \eta \rangle$ on $\partial \Sigma$. Therefore, applying Picone's identity \eqref{picone} with $\phi = \varphi W\|\nabla u\|$, the boundary term in \eqref{eq_variecaccio} can be rewritten as 
%
	\[
	\begin{array}{lcl}
	\disp \int_{\partial \Sigma} \frac{\varphi^2 \bar v}{\bar v+\eps} W \langle \nabla W, \eta \rangle 
	& = & \disp \int_{\partial \Sigma} \varphi^2W^2\|\nabla u\|^2 \langle \frac{\nabla \bar v}{\bar v+\eps} , \eta \rangle \\[0.5cm]
	& = & \disp \disp \int_{\Sigma} \big\|\nabla( \varphi  W \|\nabla u\|\big)\big\|^2 - \int_\Sigma (\bar v + \eps)^2 \left\| \nabla \left( \frac{\varphi\|\nabla u\|W}{\bar v + \eps}\right) \right\|^2.
	\end{array}
	\]
Plugging into \eqref{eq_variecaccio}, we get	
	\begin{equation}\label{eq_variecaccio_2}
	\begin{array}{lcl}
	\disp \int_\Sigma q\varphi^2 \frac{\bar v}{\bar v+\eps}W^2 & = & \disp \int_{\Sigma} \big\|\nabla( \varphi  W \|\nabla u\|\big)\big\|^2 - 2 \int_\Sigma \varphi W \langle \nabla \varphi, \nabla W \rangle \\[0.5cm]
	& & \disp - \int_\Sigma \varphi^2 \|\nabla W\|^2 - \int_\Sigma \varphi^2 W \langle \frac{\eps \nabla \bar v}{(\bar v+\eps)^2}, \nabla W \rangle \\[0.5cm]
	& & \disp + 2 \int_\Sigma \varphi W \frac{\eps}{\bar v+\eps} \langle \nabla \varphi, \nabla W \rangle +  \int_\Sigma \varphi^2 \frac{\eps}{\bar v + \eps} \|\nabla W\|^2 \\[0.5cm]
	& & \disp - \int_\Sigma (\bar v + \eps)^2 \left\| \nabla \left( \frac{\varphi\|\nabla u\|W}{\bar v + \eps}\right) \right\|^2.
	\end{array}
	\end{equation}
	By the dominated convergence theorem, since $\bar v > 0$ on $\Sigma$, we get
	\[
	\int_\Sigma \left| 2\varphi W \frac{\eps}{\bar v+\eps} \langle \nabla \varphi, \nabla W \rangle + \varphi^2 \frac{\eps}{\bar v+\eps} \|\nabla W\|^2 \right| \ra 0 \qquad \text{as } \, \eps \ra 0.
	\]
	We examine the integral
	\[
	(I) : = \left |\int_\Sigma \varphi^2 W \langle \frac{\eps \nabla \bar v}{(\bar v+\eps)^2}, \nabla W \rangle \right|
	\]
The limit of $(I)$ as $\eps \to 0$ is easily seen to be zero if $\mathrm{spt} \, \varphi \cap \partial \Omega = \emptyset$, since in this case $\bar v$ has a positive lower bound. Fix then $K \subset \overline{\Omega}$ compact that meets the boundary $\partial \Omega$. We claim that there exists a constant $C_K>0$ such that
	\begin{equation}\label{bound_boundary}
	|W\langle \nabla \bar v, \nabla W \rangle| \le C_K \bar v \qquad \text{on } \, K. 
	\end{equation}
	We first consider points of $K \cap \partial \Omega$. Since $W$ is constant on $\partial \Omega$, 
	\[
	W|\langle \nabla \bar v, \nabla W \rangle | = W|\langle \nabla \bar v, \eta \rangle \langle \eta, \nabla W \rangle | \qquad \text{on } \, \partial \Omega.
	\]	
	If $\di u \neq 0$ on $\partial \Omega$, we use Lemma \ref{lem_ideboundary} to get
	\[
	\begin{array}{lcl}
	W|\langle \nabla \bar v, \nabla W \rangle | & = & \disp \frac{\langle \nabla W, \eta \rangle^2}{\|\nabla u\|^2}  \bar v \\[0.3cm]
	& = & \disp W^6\langle \nabla \|\nabla u\|, \eta \rangle^2 \bar v \le W^6 \|\nabla ^2 u\|^2 \bar v \le C_K \bar v  \qquad \text{on } \, K \cap \partial \Omega,
	\end{array}
	\]
	since $u \in C^2$ up to $\partial \Omega$, where we used the third in \eqref{nablau2} and Kato inequality $\|\nabla \|\nabla u\|\|^2 \le \|\nabla ^2 u\|^2$. On the other hand, if $\di u = 0$ on $\partial \Omega$ then 
	\[
	W\langle \nabla \bar v, \nabla W \rangle = \frac{1}{2}\langle \nabla \bar v, \nabla |Du|^2 \rangle 
	\]
	vanishes on $\partial \Omega$, so \eqref{bound_boundary} holds on $K \cap \partial \Omega$. To deduce the validity of \eqref{bound_boundary} on the entire $K$, observe that if $\bar v(x) = 0$ for some $x \in \partial \Omega$, then $\nabla \bar v(x) \neq 0$ by Hopf boundary point Lemma. This observation and the positivity of $\bar v$ on $\Omega$ imply \eqref{bound_boundary}. Concluding, by the compactness of the support of $\varphi$, there exists $C(u,\varphi)$ such that $|W\langle \nabla \bar v, \nabla W \rangle| \le C \bar v$ on ${\rm spt } \varphi$, and by Lebesgue convergence theorem we get
	\[
	(I) \le \int_\Sigma \varphi^2 \frac{C \eps}{\bar v+\eps} \ra 0 \qquad \text{as } \, \eps \ra 0.
	\]	
Because of Fatou's Lemma, 
	\[
	\int_\Sigma \bar v^2 \left\| \nabla\left( \frac{\varphi\|\nabla u\|W}{\bar v}\right) \right\|^2 \le \liminf_{\eps \ra 0} \int_\Sigma (\bar v + \eps)^2 \left\| \nabla \left( \frac{\varphi\|\nabla u\|W}{\bar v + \eps} \right) \right\|^2
	\]
	so letting $\eps \ra 0$ in \eqref{eq_variecaccio_2} we obtain
	\begin{equation}\label{eq_variecaccio_3}
	\begin{array}{l}
	\disp \int_\Sigma q\varphi^2W^2 + \int_\Sigma \bar v^2 \left\| \nabla \left( \frac{\varphi\|\nabla u\|W }{\bar v}\right) \right\|^2 \\[0.5cm]
\le \disp - 2 \int_\Sigma \varphi W \langle \nabla \varphi, \nabla W \rangle + \int_{\Sigma} \big\|\nabla( \varphi  W \|\nabla u\|\big)\big\|^2 \disp - \int_\Sigma \varphi^2 \|\nabla W\|^2.
	\end{array}
	\end{equation}
	Denote with (III) the right-hand side of \eqref{eq_variecaccio_3}. Using $W \|\nabla u\| = \sqrt{W^2-1}$, we compute
	\[
	\begin{array}{lcl}	
	({\rm III}) & = & \disp - \int_\Sigma \varphi \langle \nabla \varphi, \nabla(W^2-1) \rangle + \int_\Sigma \|\nabla( \varphi \sqrt{W^2-1} )\|^2 - \int_\Sigma \varphi^2 \|\nabla W\|^2 \\[0.5cm]
	& = & \disp - \int \varphi \langle \nabla \varphi, \nabla (W^2-1) \rangle + \int_{\Sigma} (W^2-1)\|\nabla \varphi\|^2 \\[0.5cm]
	& & \disp + \int_\Sigma \varphi^2 \big[ \| \nabla \sqrt{W^2-1} \|^2 - \|\nabla W\|^2 \big] + 2 \int_\Sigma \varphi \sqrt{W^2-1} \langle \nabla \varphi, \nabla \sqrt{W^2-1} \rangle \\[0.5cm]
	& = & \disp \int (W^2-1)\|\nabla \varphi\|^2 + \int_\Sigma \varphi^2 \big[ \| \nabla \sqrt{W^2-1} \|^2 - \|\nabla W\|^2 \big].
	\end{array}
	\]	
	Differentiating the third in \eqref{nablau2} we get $\|\nabla W\|^2 = W^6 \|\nabla u\|^2 \|\nabla \|\nabla u\|\|^2$, and therefore 
	\[
	\| \nabla \sqrt{W^2-1} \|^2 - \|\nabla W\|^2 = \frac{\|\nabla W\|^2}{W^2-1} = W^4 \|\nabla \|\nabla u\|\|^2
	\]	
	Thus
	\[
	\begin{array}{lcl}	
	({\rm III}) & = & \disp  \int_\Sigma (W^2-1)\|\nabla \varphi\|^2 + \int_\Sigma \varphi^2 W^4 \|\nabla \|\nabla u\|\|^2.
	\end{array}
	\]
	Putting this into \eqref{eq_variecaccio_3} we get
	$$
		\disp \int_\Sigma \varphi^2W^2(q-W^2\|\nabla \|\nabla u\|\|^2) + \int_\Sigma \bar v^2 \left\| \nabla \left( \frac{\varphi\|\nabla u\|W }{\bar v}\right) \right\|^2 \leq \int_\Sigma (W^2-1)\|\nabla \varphi\|^2
	$$
	Since
	$$
		q = \|\II_\Sigma\|^2 + \overline{\Ricc}({\bf n}, {\bf n}) = \|\II_\Sigma\|^2 + W^{-2}\Ricc(Du,Du)
	$$
	we can further rewrite
	\begin{align*}
		\int_\Sigma \varphi^2 & \left[W^2(\|\II_\Sigma\|^2-W^2\|\nabla \|\nabla u\|\|^2) + \Ricc(Du,Du)\right] + \\
		& \qquad \qquad + \int_\Sigma \bar v^2 \left\| \nabla \left( \frac{\varphi\|\nabla u\|W }{\bar v}\right) \right\|^2 \leq \int_\Sigma (W^2-1)\|\nabla \varphi\|^2
	\end{align*}
	By \eqref{hessian_on_graph_u} and \eqref{eq_stezum},
	\[
	\begin{array}{lcl}
		\|\II_\Sigma\|^2 - W^2 \|\nabla \|\nabla u\|\|^2 & = & W^2\left(\|\nabla^2 u\|^2 - \|\nabla \|\nabla u\|\|^2\right) \\[0.5cm]
		& = & W^2 \left(\|\nabla_\top \|\nabla u \|\|^2 + \|\nabla u\|^2 \|A\|^2\right).		
	\end{array}
	\]
	Plugging into \eqref{eq_variecaccio_3}, rearranging and using $\di x_g = W^2 \di x_W$ and $(W^2-1)/W^2 = \|\nabla u\|^2$ we conclude \eqref{eq_poinc}. 
\end{proof}

We are ready for the 

\begin{proof}[Proof of Proposition \ref{prop_splitting}]

	By the parabolicity of $\overline \Omega$, every compact subset $K\subseteq\overline{\Omega}$ has zero capacity in the manifold with boundary $(\overline{\Omega},\sigma)$, that is,
	$$
		\inf\left\{ \int_\Omega |D\phi|^2 \di x : \phi\in \lip_c(\overline{\Omega}), \, \phi \geq 1 \, \text{ on } \, K \right\} = 0.
	$$
In particular, by \cite[Thm. 1.5]{imperapigolasetti}, there exists a sequence $\{\varphi_j\} \subset \lip_c(\overline{\Omega})$ satisfying
	\[
		\varphi_j \ra 1 \ \ \text{ in } \, W^{1,\infty}_\loc(\overline{\Omega}), \qquad \int_\Omega |D\varphi_j|^2 \ra 0
	\] 
	as $j \ra \infty$. For each $j\geq 1$, we apply Lemma \ref{lem_poinc} to deduce
	\[
	\begin{array}{l}
		\disp \int_\Sigma \left[ W^2\left( \|\nabla_\top \|\nabla u\|\|^2 + \|\nabla u\|^2 \|A\|^2 \right) + \frac{\Ricc(Du,Du)}{W^2} \right] \varphi_j^2 \di x_g + \\[0.5cm]
		\disp \qquad \qquad + \int_\Sigma \frac{\bar v^2}{W^2} \left\| \nabla \left(\frac{\varphi_j\|\nabla u\|W}{\bar v}\right) \right\|^2 \di x_g \le \int_\Sigma \|\nabla u\|^2 \|\nabla \varphi_j \|^2 \di x_g.
	\end{array}
	\]
The assumed boundedness of $|Du|$ guarantees that there exists a constant $C_0$ such that $W \le C_0$ on $\Omega$. Hence, 
	\[
		\int_\Sigma \|\nabla u\|^2 \|\nabla \varphi_j\|^2 \di x_g \le \int_\Omega |D\varphi_j|^2 W \di x \le C_0 \int_\Omega |D\varphi_j|^2 \di x \to 0,
	\]
	thus letting $j \ra \infty$ and using Fatou's Lemma we deduce 
	\begin{equation}\label{eq_persplitting}
	\|\nabla_\top \|\nabla u\|\|^2 + \|\nabla u\|^2 \|A\|^2 \equiv 0, \qquad \nabla\left( \frac{\|\nabla u\|W}{\bar v} \right) \equiv 0 \qquad \text{on } \, \Sigma.
	\end{equation}
	
	
	From now on, the splitting proceeds similarly to \cite{fmv}, with a few extra topological arguments. First, from $\bar v > 0$ we deduce $\di u \neq 0$ at every point of $\Omega$. For $x \in \Sigma$, define $\nu = \nabla u/\|\nabla u\|$ and let $\{e_\alpha\}$ be an orthonormal frame tangent to $\{u = u(x)\}$ in $\Sigma$. Because of the identities 
	\[
	\langle \nabla \|\nabla u\|, e_j \rangle  = \nabla^2 u(\nu, e_j), \qquad A_{\alpha\beta} = \frac{\nabla^2_{\alpha\beta} u}{\|\nabla u\|},
	\]
with $1 \le j \le m$ and $A$ the second fundamental form of the hypersurface $\{ u = u(x)\}$ in $\Sigma$, the first in \eqref{eq_persplitting} implies that the only nonzero component of $\nabla^2 u$ (hence, by \eqref{hessian_on_graph_u}, of $D^2u$), is the one in the direction of $Du$:
	\begin{equation}\label{eq_D2u}
	D^2 u = D^2 u\left( \frac{Du}{|Du|},\frac{Du}{|Du|} \right) \frac{\di u}{|Du|} \otimes \frac{\di u}{|Du|} \qquad \text{on } \, \Omega.
	\end{equation}
A straightforward computation allows us to deduce from \eqref{eq_D2u} that $|Du|$ is locally constant on level sets of $u$, that integral curves of $Du /|Du|$ are geodesics in $M$, and that level sets of $u$ are totally geodesic both in $\Sigma$ and in $\Omega$. In the limit, we infer that each component of $\partial \Omega$ is totally geodesic. Let $N \subset \Omega$ be a connected component of a level set of $u$, say of $\{u = b\}$ with $b \not\in \{b_1,b_2, \ldots, b_{j_0}\}$. By the implicit function theorem, note that $N$ is properly embedded both in $\Omega$ and in $M$, and is therefore a complete manifold without boundary. We denote with $\Phi(t,x)$ the flow of $D u / |Du|$ starting from $N$, defined on the connected set
	\[
	\mathscr{D} \subset \R \times N, \qquad \mathscr{D} = \big\{ (t,x) : x \in N, \ t \in (t_1(x), t_2(x))\big\},
	\]
where, for every $x\in N$, $t_1(x)\in[-\infty,0)$ and $t_2(x)\in(0,+\infty]$ are the extrema of the largest open interval $I_x = (t_1(x),t_2(x))$ such that for every $t\in I_x$ the point $\Phi(t,x)$ is well defined and belongs to $\Omega$. If $t_1(x) > -\infty$ (respectively, if $t_2(x) < +\infty$) then the curve $t \mapsto \Phi(t,x)$ converges to a point of $\partial\Omega$ as $t \searrow t_1(x)$ (resp., $t\nearrow t_2(x)$) which we shall denote as $x_\ast = \Phi(t_1(x)^+,x)$ (resp., $x^\ast = \Phi(t_2(x)^-,x)$). The function $t_1$ is upper semi-continuous on $N$, that is, for every $x\in N$ we have
$$
	\limsup_{n\to \infty} t_1(x_n) \leq t_1(x)
$$
for every sequence $\{x_n\} \subseteq N$ converging to $x$: otherwise, we could find $t\in(t_1(x),0]$ and a sequence $\{x_n\}$ converging to $x$ such that $t_1(x_n) \to t$, yielding $\partial\Omega \ni (x_n)_\ast \to \Phi(t,x) \in \Omega$, absurd. Similarly, the function $t_2$ is lower semi-continuous on $N$. Hence, $\mathscr{D}$ is open in $\R \times N$. From \eqref{eq_D2u} we deduce
	\[
	\left( D_V \frac{Du}{|Du|}, W \right) = 0 \qquad \forall \, V,W \in T\Omega,
	\]
thus $Du /|Du|$ is a parallel vector field.	By standard theory, the induced metric on $\mathscr{D}$ by $\Phi$ is the product metric $\di t^2 + \sigma_{|N}$. Let $c_0>0$ be the constant value of $|Du|$ on $N$ and let $\beta$ be the maximal solution of the Cauchy problem
\begin{equation}\label{eq_beta_1}
	\begin{cases}
		\beta' = H \dfrac{(1+\beta^2)^{3/2}}{\beta}, \\
		\beta(b) = c_0.
	\end{cases}
\end{equation}
Since $u$ is strictly increasing along the curves $t \mapsto \Phi(t,x)$ and $|Du|$ is locally constant on level sets of $u$, for every $x\in\Omega$ there exist a neighbourhood $U_x \subseteq \Omega$ and a smooth real function $\beta_x$ such that
$$
	|Du| = \beta_x(u) \qquad \text{on } \, U_x.
$$
Since $Du / |Du|$ is parallel, on $U_x$ we have
\begin{align*}
	H & = \diver\left( \frac{Du}{\sqrt{1+|Du|^2}} \right) = \diver\left( \frac{|Du|}{\sqrt{1+|Du|^2}} \frac{Du}{|Du|} \right) = D_{Du/|Du|}\frac{|Du|}{\sqrt{1+|Du|^2}} \\
	& = \beta_x(u) \left( \frac{\beta_x}{\sqrt{1+\beta_x^2}} \right)'(u) = \frac{\beta_x(u)\beta_x'(u)}{(1+\beta_x(u)^2)^{3/2}}
\end{align*}
that is, $\beta_x$ is a solution of the Cauchy problem
$$
	\begin{cases}
		y' = H \dfrac{(1+y^2)^{3/2}}{y}, \\
		y(u(x)) = |Du(x)|.
	\end{cases}
$$
Without loss of generality, let us suppose that $\beta_x$ is the maximal solution of this problem. For $x\in N$, by uniqueness $\beta_x = \beta$, hence for every $x_1,x_2\in\Phi(\mathscr{D})$ belonging to the same curve $t \mapsto \Phi(t,x)$, $x\in N$, it must hold $\beta_{x_1} = \beta_{x_2}$. Therefore, $\beta_x = \beta$ for every $x\in\Phi(\mathscr{D})$, equivalently
$$
	|Du| = \beta(u) \qquad \text{on } \, \Phi(\mathscr{D}).
$$

We claim that $\Phi(\mathscr{D}) = \Omega$. The map $\Phi$ is a diffeomorphism and $\mathscr{D}$ is open in $\R\times N$, so $\Phi(\mathscr{D})$ is open in $\Omega$. We check that $\Phi(\mathscr{D})$ is also closed in $\Omega$, thus deducing $\Phi(\mathscr{D}) = \Omega$ by connectedness of $\Omega$.



First, we prove that $t_1$ and $t_2$ are constant on $N$. We show this for $t_1$, the proof for $t_2$ being analogous. Let us define an auxiliary function $\rho : \overline{u(\Phi(\mathscr D))} \to \R$ by setting
$$
	\rho(s) = \int_b^s \frac{\di \sigma}{\beta(\sigma)} \qquad \text{for every } \, s \in \overline{u(\Phi(\mathscr D))}.
$$
Note that $\rho$ is continuous, strictly increasing and finite-valued since the above integral converges for every $s\in \overline{u(\Phi(\mathscr D))}$. Indeed, $\beta$ is continuous and strictly positive on $u(\Phi(\mathscr D))$,  and if $s\in\overline{u(\Phi(\mathscr D))}$ is such that $\beta(\sigma)\to 0$ as $\sigma\to s$, then by \eqref{eq_beta_1} necessarily $H\neq0$ and $\beta(\sigma) \sim \sqrt{2|H||\sigma-s|}$ as $\sigma\to s$, so $1/\beta$ is integrable in a neighbourhood of $s$. Also note that by integrating $\frac{\di}{\di t} u(\Phi) = |D u|(\Phi) = \beta(u(\Phi))$ we get
\[
	t = \int_{b}^{u(\Phi(t,x))} \frac{\di \sigma}{\beta(\sigma)} = \rho(u(\Phi(t,x))) \qquad \text{for every } \, (t,x) \in \mathscr{D}.
\]
We show that $t_1$ is lower semi-continuous on $N$. Suppose, by contradiction, that for some $x\in N$ and for some sequence $\{x_n\}\subseteq N$ converging to $x$ we have
$$
	\lim_{n\to \infty} t_1(x_n) < t_1(x).
$$
Fix $\bar t$ such that
$$
	\lim_{n\to \infty} t_1(x_n) < \bar t < t_1(x), \qquad \rho^{-1}(\bar t) \not\in \{b_1, b_2,\dots \}.
$$
Then, $\{\Phi(\bar t,x_n)\} \subseteq \Omega$ converges to a point $\bar x$ of $\partial\Omega$. Along this sequence, $u$ has the constant value $\rho^{-1}(\bar t)$, so by continuity it must be $u(\bar x) = \rho^{-1}(\bar t)$. But $\rho^{-1}(\bar t) \not\in \{b_1,b_2,\dots\} = u(\partial\Omega)$ and we have reached a contradiction. Since we already showed that $t_1$ is upper semi-continuous, we conclude that $t_1$ is continuous on $N$. For every $x\in N$, we either have $t_1(x) = -\infty$ or $t_1(x) \in (-\infty,0)$. In the second case, the endpoint $x_\ast = \lim_{t\to t_1(x)+} \Phi(t,x)$ belongs to $\partial\Omega$ and by continuity $t_1(x) = \rho(u(x_\ast))$. So, $t_1(N) \subseteq \{\rho(b_1),\rho(b_2),\dots\} \cup \{-\infty\}$. Since this set contains no open intervals and $t_1$ is continuous on the connected set $N$, we conclude that $t_1$ is constant.

Let $T_1\in[-\infty,0)$ and $T_2 \in (0,+\infty]$ be the constant values of $t_1$ and $t_2$ on $N$, so that
$$
	\mathscr{D} = (T_1,T_2) \times N.
$$
For every $\bar t \in (T_1,T_2)$, the image $N_{\bar t} = \Phi(\{\bar t\}\times N) \subseteq \Omega$ is a connected open subset of the embedded submanifold $\{u = \rho^{-1}(\bar t)\} \subseteq \Omega$. The restriction $\Phi_{|\{\bar t\} \times N} : \{\bar t\} \times N \to N_{\bar t}$ is a local Riemannian isometry and $\{\bar t\} \times N$ is complete, so $\Phi_{|\{\bar t\} \times N}$ is a Riemannian covering map and therefore $N_{\bar t}$ is also complete with respect to its intrisinc geodesic distance, that we shall denote by $d_{\bar t}$ (see \cite[Lemma 5.6.4 and Proposition 5.6.3]{petersen}).


We prove that $\Phi(\mathscr{D})$ is closed in $\Omega$. 
Let $\{p_n\}\subseteq \Phi(\mathscr{D})$ be a given sequence converging to some point $\bar p \in \Omega$. We have to show that $\bar p \in \Phi(\mathscr{D})$. Set $\bar t = \rho(u(\bar p))$, and observe that $u(\bar p) \in u(\overline{\Phi(\mathscr D)}) \subseteq \overline{u(\Phi(\mathscr D))}$, hence $\bar t$ is finite. For every $n$ we can find $(t_n,x_n) \in \mathscr{D}$ such that $p_n = \Phi(t_n,x_n)$. By continuity, $t_n = \rho(u(p_n)) \to \rho(u(\bar p)) = \bar t$, hence $T_1 \leq \bar t \leq T_2$. Both inequalities are strict, otherwise either $\{(x_n)_\ast\} = \{\Phi(T_1^+,x_n)\}$ or $\{(x_n)^\ast\} = \{\Phi(T_2^-,x_n))\}$ would be a sequence of points of $\partial\Omega$ converging to $\bar p\in\Omega$, absurd. Setting $q_n = \Phi(\bar t,x_n)$ for every $n$, we have that $\{q_n\}$ is a sequence of points of $N_{\bar t}$ converging to $\bar p$ in $M$, since $d_\sigma(p_n,q_n) \leq |\bar t - t_n| \to 0$. Hence, $\{q_n\}$ is a Cauchy sequence in $M$. By completeness of $M$, any two points $q_n$, $q_{n'}$ are joined by a minimizing geodesic arc in $M$. Since $(N_{\bar t},d_{\bar t})$ is complete and totally geodesic, every geodesic in $M$ joining two points of $N_{\bar t}$ must lie in $N_{\bar t}$. So, $\{q_n\}$ is a Cauchy sequence in $N_{\bar t}$ and therefore converges to some point $\bar q \in N_{\bar t}$. Since $N_{\bar t}$ is embedded in $M$, $\bar q = \bar p$ and we conclude $\bar p \in N_{\bar t} \subseteq \Phi(\mathscr{D})$, as desired. This shows that $\Phi(\mathscr{D})$ is closed in $\Omega$.

As already stated, since $\Phi(\mathscr{D})$ is non-empty and both open and closed in the connected set $\Omega$, we have $\Phi(\mathscr{D}) = \Omega$. Thus, $\Phi$ realizes an isometry between $\Omega$ and the product manifold
$$
	(T_1,T_2)\times N,
$$
and $N$ is parabolic because of Lemma \ref{lem_paraprod}. Furthermore, $u$ only depends on the variable $t$ because 
$$
	u(\Phi(t,x)) = \rho^{-1}(t) \qquad \text{for every } \, (t,x) \in (T_1,T_2)\times N.
$$
In the chart $\Phi$, $u= u(t)$ is therefore a solution of 
	\[
	H = \diver \left( \frac{Du}{\sqrt{1+ |Du|^2}} \right) = \diver \left( \frac{u'\partial_t}{\sqrt{1+ (u')^2}} \right) = \left( \frac{u'}{\sqrt{1+ (u')^2}} \right)'.
	\]
	Integration of the ODE shows that the possibility $(-\infty, T_2) \times N$ is incompatible with the fact that $u$ is increasing and bounded from below, while the other models lead to the solutions $u_{c_1,H}$ in \eqref{eq_split}, up to reparametrizing $t \mapsto t - T_1$ and setting $T = T_2 - T_1$. To check $(ii)$, the second in \eqref{eq_persplitting} implies 
	\[
	\bar v = c W\|\nabla u\| \qquad \text{on } \, \Omega,
	\]	
for some constant $c>0$. Using \eqref{nablau2}, the identity rewrites as $(X, Du) = c|Du|$, that is, $(X, \partial_t) = c$.
\end{proof}

\begin{proof}[Proof of Theorem \ref{thm-epigraph}]
For $i=1,2$, let $\Gamma_i = \{ (\tau,x) \in M : \tau = \varphi_i(x) \}$ be the graph of $\varphi_i$ over $N$. The boundary $\partial\Omega$ is either $\Gamma_1$, in case $(i)$, or $\Gamma_1\cup\Gamma_2$, in case $(ii)$. The unit outward-pointing normal $\bar\eta$ on $\partial\Omega$ is given by
	\begin{equation}
		\bar\eta = \frac{D\varphi_1 - \partial_\tau}{\sqrt{1+|D\varphi_1|^2}} \qquad \text{on } \, \Gamma_1, \qquad \qquad \bar\eta = \frac{\partial_\tau - D\varphi_2}{\sqrt{1+|D\varphi_2|^2}} \qquad \text{on } \, \Gamma_2,
	\end{equation}
	where $D\varphi_i$ is the gradient of $\varphi_i$ in $N$. In case $(i)$ we have $c(\partial_\tau,\bar\eta) = -c/\sqrt{1+|D\varphi_1|^2}$ on $\Gamma_1 = \partial\Omega$: this quantity is nonzero at every point of $\Gamma_1$, so for either $X = \partial_\tau$ or $X = -\partial_\tau$ we have $c(X,\bar\eta) > 0$ on $\partial\Omega$. In case $(ii)$ we have $c_1(\partial_\tau,\bar\eta) = -c_1/\sqrt{1+|D\varphi_1|^2} \geq 0$ on $\Gamma_1$ and $c_2(\partial_\tau,\bar\eta) = c_2/\sqrt{1+|D\varphi_2|^2} \geq 0$ on $\Gamma_2$, and at east one of them does not vanish. Fix $o_N\in N$ and let $o = (0,o_N) \in \R \times N = M$. We check that $\Omega$ satisfies
	\begin{equation} \label{cor-partOm}
		\liminf_{r\to \infty} \frac{\log|\partial\Omega \cap B_r|}{r^2} < \infty, \qquad \int^{\infty} \frac{r\di r}{|\Omega \cap B_r|} = \infty
	\end{equation}
	where $B_r = B^M_r(o)$. For every $r>0$ we have
	$$
		\{ (t,x) \in B_r : t = \varphi_i(x) \} \subseteq \{ (t,x) \in \R \times B^N_r(o_N) : t = \varphi_i(x) \} \qquad \text{for } \, i = 1,2.
	$$
	Since $\varphi_1$ and $\varphi_2$ are globally Lipschitz, there exists $C>0$ such that, for every $r>0$,
	$$
		|\{ (\tau,x) \in \R \times B^N_r(o_N) : \tau = \varphi_i(x) \}| \leq C|B^N_r(o_N)| \qquad \text{for } \, i = 1,2.
	$$
	Hence, $|\partial\Omega \cap B_r| \leq 2C|B^N_r(o_N)|$ for every $r>0$ and the first condition in \eqref{cor-partOm} is satisfied under the weaker assumption \eqref{ipo_epi_weak}.
	
	Since $\varphi_1$ and $\varphi_2$ are Lipschitz, there exist $C_1 > 0$, $C_2 \geq 0$ such that
	$$
		|\varphi_i(x)| \leq C_1 + C_2 d_N(o_N,x) \qquad \text{for every } \, x \in N, \, \text{ for } \, i = 1,2.
	$$
	If $\Omega = \{ (\tau,x) \in M : \tau > \varphi_1(x) \}$, then for each $r>0$ the inclusion $\Omega \cap B_r \subseteq ( - C_1 -C_2 r, r) \times B^N_r(o)$ implies the inequality
	$$
		|\Omega \cap B_r| \leq (C_1 + (1+C_2)r)|B^N_r(o)|.
	$$
	If $\Omega = \{ (\tau,x) \in M : \varphi_1(x) < \tau < \varphi_2(x) \}$ then $\Omega \cap B_r \subseteq (-C_1-C_2 r,C_1+C_2 r) \times B^N_r(o)$ and
	$$
		|\Omega \cap B_r| \leq 2(C_1+C_2 r)|B^N_r(o)|.
	$$
	In both of the cases, if \eqref{ipo_epigraph} holds then 
	\begin{equation} \label{Om-Br-log}
		\limsup_{r\to \infty} \frac{|\Omega \cap B_r|}{r^2 \log r} < \infty
	\end{equation}
	and the second in \eqref{cor-partOm} follows. Furthermore, in the second case if $\varphi_1$ and $\varphi_2$ are bounded then we can choose $C_2 = 0$ and we obtain the validity of \eqref{Om-Br-log} under the weaker condition \eqref{ipo_epi_weak}.
	
	We are now in position to apply Theorem \ref{teo_splitting}. The domain $\Omega$ splits isometrically as the Riemannian product 
	\[
	\big((0,T) \times \Gamma_1, \di t^2 + \sigma_{|\Gamma_1}\big)
	\] 
with $T= \infty$ in case $(i)$, or $T\in(0,\infty)$ in case $(ii)$. Furthermore, $c_1 \le 0$ ($c_1 < 0$ if $H \le 0$) and the function $u$ coincides with $u_{c,H}(t)$ or $u_{c_1,H}(t)$, according to whether $(i)$ or $(ii)$ is considered. In particular, $\Gamma_1$ is a totally geodesic graph and, in case $(ii)$, $\Gamma_2$ is obtained by translating $\Gamma_1$ in the direction of $Du/|Du|$. The graph $\Gamma_1$ is totally geodesic in $M = \R\times N$ if and only if $D^2 \varphi_1 \equiv 0$ on $N$. If $\varphi_1$ is constant, say $\varphi_1 \equiv a_1$, then $t = \tau -a_1$. If $\varphi_1$ is not constant, then $D\varphi_1$ is a parallel, nowhere vanishing vector field on $N$, with constant norm $a_0 := |D\varphi_1| > 0$. Level sets of $\varphi_1$ in $N$ are totally geodesic and equidistant, so $N$ splits as a Riemannian product $N = \R\times N_0$ for some complete, boundaryless manifold $N_0$ with $\Ricc \ge 0$. The second in \eqref{ipo_epigraph} and a computation similar to the one yielding to \eqref{Om-Br-log} implies
	\begin{equation}\label{eq_perCalYau}
	\limsup_{r \to \infty} \frac{|B_r^{N_0}|}{\log r} < \infty.
	\end{equation}
If $N_0$ were non-compact, an inequality due to E.Calabi and S.T.Yau would ensure that $|B_r^{N_0}| \ge Cr$ for $r \ge 1$ and some constant $C$, contradicting \eqref{eq_perCalYau}. Hence, $N_0$ must be compact.\par
	Regarding the function $\varphi_1$, up to a reparametrization $s \mapsto -s$, $\partial_s\varphi_1 \equiv |D\varphi_1| = a_0$. Hence, for some constant $a_1\in\R$ we have
	$$
		\varphi_1(s,\xi) = a_0 s + a_1 \qquad \text{for every } \, (s,\xi) \in \R \times N_0.
	$$
	The coordinate function $s$ on $N$ extends to a smooth function on the product $M = \R\times N$. Regarding $M$ as the product $\R \times \R \times N_0$, we can denote its generic point by $(\tau,s,\xi)$, with $\xi\in N_0$. On $\Gamma_1$ we have $\bar\eta = -\partial_t$, therefore
	$$
		a_0 \partial_s = D\varphi_1 = -\sqrt{1+a_0^2}\partial_t + \partial_\tau.
	$$
Integrating, 
	$$
		t = \frac{\tau - a_0 s - a_1}{\sqrt{1+a_0^2}} + C \qquad \text{on } \, \Omega,
	$$
for some constant $C\in\R$. Since $t = 0$ and $\tau = \varphi_1 = a_0 s + a_1$ on $\Gamma_1$, we conclude $C=0$ and we obtain the desired expressions for $u$ in terms of $(\tau,s,\xi)$.
\end{proof}



\begin{appendix}

\section{Appendix}

\begin{proposition}\label{prop_AC}
	Let $H\in\R$, $\kappa\geq0$, $C\geq 0$ satisfy
	\begin{align}
		\label{appr-hp1-app}
		\frac{H^2}{m} + C^2 - (m-1)\kappa^2 \geq 0 & \qquad \text{if} \quad H \leq 0,
		\\
		\label{appr-hp2-app}
		\frac{H^2}{m} + C^2 - (m-1)\kappa^2 > 0 & \qquad \text{if} \quad H > 0.
	\end{align}
	Then, there exists $A\geq 1$ such that
	\begin{equation} \label{CAHge0-app}
		\frac{H^2}{m} - \frac{CH}{t} + \left( C^2 - (m-1)\kappa^2 \right) \frac{t^2-1}{t^2} \geq 0 \qquad \text{for every } \, t \geq A.
	\end{equation}
	In particular,
	\begin{itemize}
	\item [-] if $H \leq 0$ then \eqref{CAHge0-app} is true for any $A\geq 1$
	\item [-] if $H > 0$ and $C \geq \sqrt{m-1}\kappa$ then \eqref{CAHge0-app} holds for a given $A\geq 1$ if and only if
	$$
	\frac{H^2}{m} - \frac{CH}{A} + \left( C^2 - (m-1)\kappa^2 \right) \frac{A^2-1}{A^2} \geq 0
	$$
	\item [-] if $H > 0$ and $C < \sqrt{m-1}\kappa$ then \eqref{CAHge0-app} holds for a given $A\geq 1$ if
	$$
	\left\{\begin{array}{l}
	\dfrac{H^2}{m} - \dfrac{CH}{A} + \left( C^2 - (m-1)\kappa^2 \right) 	\dfrac{A^2-1}{A^2} \geq 0, \\[0.2cm]
	CHA + 2\left( C^2 - (m-1)\kappa^2 \right) \geq 0
	\end{array}\right.
	$$
	and these conditions are also necessary, 
	unless
	\begin{equation} \label{discr<0-app}
		\sqrt{1+\frac{m}{4}}C < \sqrt{m-1}\kappa, \qquad \frac{H^2}{m} \geq \frac{\left( (m-1)\kappa^2 - C^2 \right)^2}{(m-1)\kappa^2 - \left(1+\frac{m}{4}\right)C^2},
	\end{equation}
	in which case \eqref{CAHge0-app} is satisfied for any $A\geq 1$.
\end{itemize}
\end{proposition}

\begin{proof}
	Set
	$$
		P(s) := \frac{H^2}{m} - CHs + \left( C^2 - (m-1)\kappa^2 \right) \left( 1 - s^2 \right) \qquad \text{for every } \, s \in \R.
	$$
	Note that $P(0) \geq 0$ under assumptions \eqref{appr-hp1-app}-\eqref{appr-hp2-app}. By substituting $s=1/t$, for any given $A\geq 1$ condition \eqref{CAHge0-app} is equivalent to requiring that $P(s) \geq 0$ for every $0 < s \leq 1/A$.
	\begin{itemize}
		\item [-] If \eqref{appr-hp1-app} holds then
		$$
			P(s) \geq \frac{H^2}{m} + \left( C^2 - (m-1)\kappa^2 \right) \left( 1 - s^2 \right) \geq 0 \qquad \text{for every } \, 0 \leq s \leq 1.
		$$
		\item [-] If $H>0$ and $C\geq\sqrt{m-1}\kappa$ then $P'(s) = -CH+2\left( C^2 - (m-1)\kappa^2 \right)s \leq 0$ for $s\geq0$, so if $P(1/A) \geq 0$ then $P(s) \geq P(1/A) \geq 0$ for every $0 \leq s \leq 1/A$.
		\item [-] If $H<0$ and $C < \sqrt{m-1}\kappa$ then setting $a = (m-1)\kappa^2 - C^2 > 0$ we can write
		$$
			P(s) = as^2 - CHs + \frac{H^2}{m} - a.
		$$
		$P$ is a quadratic polynomial with positive leading coefficient and $P'(s) = -CH - 2as$. Note that $P(0) > 0$, $P'(0) \leq 0$ since we are in case \eqref{appr-hp2-app}, so $P$ can attain negative values only in $\R^+$. If $P(1/A) \geq 0$ and $P'(1/A) \leq 0$ for some $A>0$, then $P(s) \geq 0$ for $0 \leq s \leq 1/A$. The polynomial $P$ has discriminant
		$$
			\Delta_P = C^2H^2 + 4a^2 - 4a\frac{H^2}{m} = 4\left[ a^2 - \left( a - \frac{m}{4}C^2 \right)\frac{H^2}{m} \right].
		$$
		If $\Delta_P > 0$, then conditions $P(1/A) \geq 0$, $P'(1/A) \leq 0$ are also necessary to characterize $A>0$ such that \eqref{CAHge0-app} holds. Otherwise, \eqref{CAHge0-app} is satisfied for any $A>0$ since $P \geq 0$ on $\R$. A necessary condition for $\Delta_P \leq 0$ is clearly $a > mC^2/4$, and if this holds then $\Delta_P \leq 0$ if and only if $H^2/m \geq a^2/(a-mC^2/4)$. Explicitating $a$, these requests read as \eqref{discr<0-app}.
	\end{itemize}
\end{proof}

\begin{proposition}[Lemma \ref{nuovo-lemma-appr}] \label{lem_appr_appendix}
	Let $m\geq 2$, $\kappa\geq 0$, $H\in\R$, $C\geq 0$ satisfy \eqref{appr-hp1-app}, \eqref{appr-hp2-app}, \eqref{CAHge0-app}.
	Then, for any $\eps>0$ there exist $A < A_1 < A+\eps$, $C < C_2 < C_1 < C+\eps$ such that
	\begin{equation} \label{appr_CAH-app}
	\inf\left\{ \frac{H^2}{m} - \frac{C_1 H}{t} + \left( C_2^2 - (m-1)\kappa^2 \right)\frac{t^2-1}{t^2} : t \geq A_1 \right\} > 0.
	\end{equation}
\end{proposition}

\begin{proof}
	We set
	$$
	P(s) := \frac{H^2}{m} - CH s + \left( C^2 - (m-1)\kappa^2 \right) (1-s^2) \qquad \text{for every } \, s \in \R.
	$$
	For any given $C_1, C_2 \in \R$, we also set
	$$
	Q(s) = Q_{C_1,C_2}(s) := \frac{H^2}{m} - C_1 H s + \left( C_2^2 - (m-1)\kappa^2 \right) (1-s^2) \qquad \text{for every } \, s \in \R.
	$$
	By substituting $s = 1/t$, \eqref{CAHge0-app} and \eqref{appr_CAH-app} are respectively equivalent to
	\begin{equation} \label{PQs}
	(i) \; \; \inf\{ P(s) : 0 < s \leq 1/A \} \geq 0, \qquad (ii) \; \; \inf\{ Q_{C_1,C_2}(s) : 0 < s \leq 1/A_1 \} > 0.
	\end{equation}
	Also observe that
	\begin{equation} \label{Q-Ps}
	Q(s) - P(s) = (C-C_1)Hs + \left( C_2^2 - C^2 \right)(1-s^2) \qquad \text{for every } s \in \R.
	\end{equation}
	
	If $H\leq 0$, then for any $C_1 > C_2 > C$ and $A_1 > 1$ we have
	$$
	Q_{C_1,C_2}(s) - P(s) \geq \left( C_2^2 - C^2 \right)(1-1/A_1^2) > 0 \qquad \text{for every } \, 0 < s \leq 1/A_1
	$$
	and then (i) $\Rightarrow$ (ii) in \eqref{PQs}. Hence, for every $\eps>0$ there exist $C_1$, $C_2$, $A_1$ as required in the statement of the lemma.
	
	We now consider the case $H>0$. We claim that for any $A_1 > A$ we have
	\begin{equation} \label{infP}
	\inf\{ P(s) : 0 < s \leq 1/A_1 \} > 0.
	\end{equation}
	Hence, for every $\eps>0$ we can find $A < A_1 < A+\eps$ such that \eqref{infP} holds. By \eqref{Q-Ps}, as $C_1 \searrow C$ the function $Q_{C_1,C_2}$ uniformly converges to $P$ on $(0,1/A_1]$. Hence, we can find $C < C_1 < C+\eps$ such that, for any choice of $C < C_2 < C_1$, condition (ii) in \eqref{PQs} is satisfied.
	
	We prove the claim. Observe that
	$$
	P'(s) = - CH - 2\left( C^2 - (m-1)\kappa^2 \right)s.
	$$
	If $C \geq \sqrt{m-1}\kappa$, then $P$ is a non-increasing function of $s\geq0$. In particular, we either have $P$ constant or $P$ strictly decreasing, and in any case $P(0) \geq H^2/m > 0$. Hence, \eqref{infP} holds for any $A_1 > A$. If $C < \sqrt{m-1}\kappa$, then $P$ is a quadratic polynomial with positive leading coefficient, attaining its global minimum for
	$$
	s = s_\ast = \frac{CH}{2\left( (m-1)\kappa^2 - C^2 \right)} \geq 0.
	$$
	Since $0 \leq C < \sqrt{m-1}\kappa$ implies $\kappa>0$, we must be in case \eqref{appr-hp2} and then
	$$
	P(0) = \frac{H^2}{m} + C^2 - (m-1)\kappa^2 > 0.
	$$
	If $P(s_\ast) > 0$, then for any $A_1 > 0$ we have \eqref{infP}. If $P(s_\ast) \leq 0$ then it must be $0 < 1/A \leq s_\ast$. Hence, $P$ is strictly decreasing on $(0,1/A]$ and for any $A_1 > A$ we have again \eqref{infP}.
\end{proof}

\end{appendix}

\bibliographystyle{amsplain}

\vspace{1cm}

\begin{flushleft}
Dipartimento di Matematica,
Universit\`a
degli studi di Milano,\\
Via Saldini 50, I-20133 Milano (Italy)\\
E-mail: giulio.colombo@unimi.it, marco.rigoli55@gmail.com
\end{flushleft} 

\begin{flushleft} Dipartimento di Matematica, Universit\`a degli Studi di Torino,\\
Via Carlo Alberto 10, I-10123 Torino (Italy)\\
E-mail: luciano.mari@unito.it (corresponding author)
\end{flushleft}

\normalsize

\end{document}